\documentclass[a4paper,12pt,reqno,oneside]{amsart} 

\usepackage{setspace} 
\usepackage{typearea} 
\usepackage{amscd,amssymb,verbatim} 
\usepackage{graphicx} \setkeys{Gin}{width=\linewidth}
\usepackage{enumitem}

\usepackage{xr-hyper}
\usepackage{cite}
\usepackage[colorlinks,backref,final,bookmarksnumbered,bookmarks]{hyperref}
\usepackage[backrefs]{amsrefs}

\usepackage{ifthen}
\newcommand{\arxiv}[2][]{\ifthenelse{\equal{#1}{}}
{\href{http://arxiv.org/abs/#2}{\tt arXiv:#2}}
{\href{http://arxiv.org/abs/math/#2}{\tt arXiv:math.#1/#2}}}

\theoremstyle{plain}
\newtheorem{theorem}{Theorem}[section]
\newtheorem{lemma}[theorem]{Lemma}
\newtheorem{corollary}[theorem]{Corollary}

\theoremstyle{definition}
\newtheorem{example}[theorem]{Example}

\newtheoremstyle{remark}
{}{}{}{}{\itshape}{}{ }{\thmname{#1}\thmnumber{ \itshape #2.}}
\theoremstyle{remark}
\newtheorem{remark}[theorem]{Remark}

\def\x{\times}
\def\but{\setminus}
\def\emb{\hookrightarrow}
\def\eps{\varepsilon}
\def\phi{\varphi}
\def\emptyset{\varnothing}
\renewcommand{\:}{\colon}
\def\xr#1{\xrightarrow{#1}}
\def\R{\mathbb{R}}
\def\Z{\mathbb{Z}}
\def\Cl#1{\overline{#1}}

\DeclareMathOperator{\id}{id}
\DeclareMathOperator{\im}{im}
\DeclareMathOperator{\sign}{sign}
\DeclareMathOperator{\lk}{lk}
\DeclareMathOperator{\cl}{cl}

\begin{document}

\title{Self $C_2$-equivalence of two-component links and invariants of link maps}
\author{Sergey A. Melikhov}
\address{Steklov Mathematical Institute of Russian Academy of Sciences,
8 Gubkina St., Moscow 119991, Russia}
\email{melikhov@mi-ras.ru}

\begin{abstract} We use Kirk's invariant of link maps $S^2\sqcup S^2\to S^4$ and its variations due to 
Koschorke and Kirk--Livingston to deduce results about classical links.
Namely, we give a new proof of the Nakanishi--Ohyama classification of two-component links in $S^3$ 
up to $\Delta$-link homotopy.
We also prove its version for string links, which is due (in a slightly different form) to Fleming--Yasuhara.
The proofs do not use Clasper Theory.
\end{abstract}

\maketitle

\section{Introduction}
Two classical links (or string links) are {\it $C_k$-equivalent} if they are related by a sequence of $C_k$-moves 
of Gusarov and Habiro, and {\it self $C_k$-equivalent} if additionally each of these $C_k$-moves involves 
strands only from one component.
Invariants of $C_{k+1}$-equivalence include all type $k$ invariants in the sense of Vassiliev \cite{Ha}, \cite{Gu},
whereas invariants of self $C_k$-equivalence include not only these, but also all type $k$ invariants in the sense 
of Kirk and Livingston (by similar arguments).
Kirk--Livingston type $k$ invariants \cite{KL} (see also \cite{M}) are those link invariants whose extension 
to singular links vanishes on all link maps with $>k$ double points (but not necessarily on all singular links 
with $>k$ double points).
They remain much more mysterious than Vassiliev finite type invariants.
In particular, Kirk and Livingston conjectured that the group of type $2$ invariants in their sense has infinite rank 
for $2$-component links with any linking number \cite{KL}. 

Self $C_1$-equivalence is better known as {\it link homotopy}.
String links are classified up to link homotopy by $\mu$-invariants with distinct indices \cite{HL}, and a link 
is link homotopic to the unlink if and only if all its $\bar\mu$-invariants with distinct indices vanish \cite{Mi1}.
There is an algorithmic classification of links up to link homotopy \cite{HL}, and classification in terms of 
explicit invariants is known for 3-component \cite{Mi1} and 4-component \cite{Le} links.

Self $C_2$-equivalence is also known as {\it $\Delta$-link homotopy} (since $C_2$-moves are also 
known as $\Delta$-moves). 
Yasuhara proved that a link is $\Delta$-link homotopic to the unlink if and only if all its $\bar\mu$-invariants 
with at most two occurrences of each index vanish \cite{Ya}.
He also proved that string links are classified up to the equivalence relation generated by $\Delta$-link homotopy 
and concordance by $\mu$-invariants in which at most two occurrences of each index vanish \cite{Ya2}.
(See also \cite{MY}*{\S6} for a classification of $2$-component string links up to the equivalence relation 
generated by self-$C_3$-equivalence and concordance.)
However, in general $\bar\mu$-invariants do not suffice even to detect self $C_k$-triviality: by a result of 
Fleming and Yasuhara \cite{FY}, the Whitehead double of the Whitehead link (which is a boundary link, so all 
its $\bar\mu$-invariants vanish) is not self $C_3$-equivalent to the unlink.
In fact, as observed by the same authors \cite{FY07}, self $C_2$-triviality of string links is already not detected 
by their $\mu$-invariants (see Example \ref{triviality} below).

In 2003, Nakanishi and Ohyama obtained a classification of $2$-component links up to $\Delta$-link homotopy \cite{NO}.
Namely, they are classified by the linking number and the generalized Sato--Levine invariant --- which are 
the first two coefficients of the power series $\nabla_L/(\nabla_{K_1}\nabla_{K_2})$, where $\nabla_L$ and 
$\nabla_{K_i}$ are the Conway polynomials of the link and of its components.
These two invariants also generate all Kirk--Livingston invariants of type $\le 1$ for two-component links \cite{KL}.
(See \S\ref{genSL} for further details on the generalized Sato--Levine invariant.)

Using Kirk's invariant of link maps $S^2\sqcup S^2\to S^4$ and its variations due to Koschorke and 
Kirk--Livingston (or rather their referee), we obtain a new proof of the Nakanishi--Ohyama theorem 
(Theorem \ref{NO}), and also of its version for string links (Theorem \ref{NOstring}); a slightly different 
string link version was obtained by Fleming and Yasuhara \cite{FY07}.
In the course of the proof we also compute the images of the three said link map invariants 
(Corollaries \ref{image-kirk}, \ref{koschorke+} and \ref{kltheorem}).
These images are known from the literature, but without proof in one case (the Kirk--Livingston invariant) and with 
accurate, but not very detailed proofs in the other two cases.
Since we actually need relative versions of these computations (Theorems \ref{image1} and \ref{image2}), 
we take the opportunity to give a more accessible exposition of all these matters, including the necessary 
background.
This exposition on link map invariants accounts for almost 2/3 of the length of the paper.
The role of link map invariants in our approach to $\Delta$-link homotopy is comparable to the role of 
Clasper Theory in the proofs by Nakanishi--Ohyama and Fleming--Yasuhara.

The point of our new proof of the Nakanishi--Ohyama theorem is that it provides
\begin{itemize}
\item a connection between $\Delta$-link homotopy in $S^3$ and link homotopy in $S^4$; 
\item an approach to classifying $n$-component string links up to $\Delta$-link homotopy.
\end{itemize}
This approach is being pursued in a subsequent work by the author.
In fact, the present note has been a part of that work in progress.
However, as the latter is growing into a longer text whose completion is competing for the author's time with 
a number of other projects, the two-component case now appears to make more sense as a self-contained 
standalone note. 
In fact, a two-page sketch of the present note was a part of a preprint privately circulated by the author in 2007.

\subsection{Sketch of the proof}

Suppose that $L,L'\:S^1_+\sqcup S^1_-\emb S^3$ are links with the same linking number $\lk(L)=\lk(L')$ and 
the same generalized Sato--Levine invariant $\beta(L)=\beta(L')$.
We want to show that they are $\Delta$-link homotopic.

Since $\lk(L)=\lk(L')$, there is a generic smooth link homotopy $h_t$ between $L$ and $L'$.
Every double point $z=h_t(p)=h_t(q)$ of this homotopy splits the singular circle $h_t(S^1_{\sigma_z})$ into 
two `lobes' $h_t(J_z)$ and $h_t(J_z')$, where $J_z$ and $J_z'$ are complementary arcs in $S^1_{\sigma_z}$ with
common endpoints $p$ and $q$.
The linking numbers $l_z$, $l'_z$ between these lobes and the other component $h_t(S^1_{-\sigma_z})$ satisfy
$l_z+l'_z=\lk(L)$.
Let us say that the double point $z$ ``algebraically cancels'' another double point $w$ of the same homotopy $h_t$
if they occur on the same component (i.e.\ $\sigma_z=\sigma_w$); the signs of $z$ and $w$ (as normal double points,
see \S\ref{basic2}) are opposite; and the unordered pairs $\{l_z,l_z'\}$ and $\{l_w,l_w'\}$ coincide.
Since $\beta(L)=\beta(L')$, it follows from a study of invariants of link maps (see Corollary \ref{realization2}) 
that $h_t$ can be amended (by inserting an appropriate self-link-homotopy of $L$) into a new link homotopy
between $L$ and $L'$ (still denoted $h_t$) whose double points can be grouped into algebraically canceling pairs.

The remainder of the proof is purely geometric (see \S\ref{geometry} for the details).
By ``delaying'' some crossing changes in $h_t$, we can modify it further so that all double points on $S^1_+$ 
occur during the time interval $[\frac15,\frac25]$ and those on $S^1_-$ occur during $[\frac35,\frac45]$, 
and each singular level of the homotopy contains precisely two double points $z$, $w$ with opposite signs and 
with $\{l_z,l_z'\}=\{l_w,l_w'\}$.
(The ``delaying'' of a crossing change in a homotopy is a straightforward geometric construction based on 
two facts: that a pair of crossing changes that are close in time ``commute'', and that each segment of 
the smooth homotopy involving no crossing changes is covered by an ambient isotopy.)
We may also assume that in fact $l_z=l_w$ and $J_z\subset J_w$ for each pair $(z,w)$ as above, since it is not 
hard to slide the endpoints of $J_z$ past those of $J_w$ in any desired way.

Since there are no nontrivial type $1$ Vassiliev invariants of knots, every knot can be undone into the unknot 
by $C_2$-moves.
Thus, by making use of the delaying technique, we may assume that $S^1_-$ is unknotted during the time interval 
$[\frac15,\frac25]$ and $S^1_+$ is unknotted during $[\frac35,\frac45]$.
Then for each pair $(z,w)$ of simultaneous double points as above, occurring at some time instant $t$,
the $h_t$-images of the two arcs $\Cl{J_w\but J_z}$ bound a Whitney disk in the complement to the other component,
$S^3\but h_t(S^1_{-\sigma_z})$.
The interior of this Whitney disk may intersect its own component $h_t(S^1_{\sigma_z})$, but if there is just one
intersection, then the effect of the two simultaneous self-intersections at $z$ and $w$ is the same as that of 
a self $C_2$-move.
If there are $n$ intersections, then the two simultaneous self-intersections can be similarly replaced by $n$ 
self $C_2$-moves. 

The proof of the string link version follows the same pattern, and in fact is considerably easier 
(in the algebraic part).

\section{Preliminaries} 

\subsection{Tangles, links and string links} \label{basic1}
A map $f$ between smooth manifolds $N$ and $M$, where $N$ is compact, is called {\it proper} if 
$f^{-1}(\partial M)=\partial N$.
By a {\it tangle} we mean a proper smooth embedding of a compact 1-manifold in a 3-manifold.
This includes {\it links} $L\:S^1\sqcup\dots\sqcup S^1\emb S^3$, tangles of the form 
$S^1\sqcup\dots\sqcup S^1\emb S^3\but\im(L)$, where $L$ is a link, 
and {\it string links} $\Lambda=\bigsqcup_i\Lambda_i\:I\sqcup\dots\sqcup I\emb I\x\R^2$,
where $\Lambda_i(j)=(j,i,0)$ for all $j\in\partial I$.
Here $I=[0,1]$ is endowed with the orientation given by the frame consisting of the tangent vector 
$\frac{\partial}{\partial t}$. 
The {\it trivial string link} $\Xi=\bigsqcup_i\Xi_i\:I\sqcup\dots\sqcup I\emb I\x\R^2$ 
is defined by $\Xi_i(t)=(t,i,0)$ for all $t\in I$.
When drawing string links, we mean $\{0\}\x\R^2$ to be on the left and $\{1\}\x\R^2$ to be on the right.
When drawing links or string links, we assume the orientation to be given by a frame $(v_1,v_2,v_3)$,
where $v_1$ and $v_2$ lie in the plane of the drawing, so that the rotation from $v_1$ to $v_2$ is counter-clockwise,
and $v_3$ points away from the reader.

\subsection{Singular tangles} \label{basic2}
A double point $z=f(x)=f(y)$ of a smooth map $f\:N\to M$ of a 1-manifold into a 3-manifold will be called {\it normal} 
if it is not a triple point, and the vectors $df_x(1)$, $df_y(1)$ in the tangent space $T_z M$ are linearly independent.
By a {\it singular tangle} we mean a proper smooth immersion of a compact 1-manifold in a 3-manifold with finitely many
double points, all of which are normal.

A double point $z=f(x,t)=f(y,t)$ of a smooth homotopy $F\:N\x I\to M$ will be called {\it normal} if 
it is not a triple point and the tangent vectors $dF_{(x,t)}(1,0)$, $dF_{(y,t)}(1,0)$ and 
$dF_{(x,t)}(0,1)-dF_{(y,t)}(0,1)$ are linearly independent.
The sign of the frame formed by these three vectors with respect to a fixed orientation of $M$ is called the {\it sign} 
of the normal double point of the homotopy.
(With the above orientation conventions, this agrees with the usual sign convention for finite type 
invariants of links and string links.)
Clearly, this sign does not change if $x$ and $y$ are interchanged.

\subsection{Link maps}
A {\it link map} $f\:N\to M$, where $M$ is a smooth manifold and $N$ is a compact smooth manifold with a fixed
presentation $N=N_1\sqcup\dots\sqcup N_m$ as a disjoint union, is a continuous proper map such that 
$f(N_i)\cap f(N_j)=\emptyset$ whenever $i\ne j$.
For our purposes, the $N_i$ will always be connected, so that the decomposition of $N$ into a disjoint union 
is automatic.
A {\it link homotopy} $h_t\:N\to M$ is a homotopy through link maps that keeps $\partial N$ fixed.
We will denote by $LM_{N\to M}$ the set of link homotopy classes of link maps $N\to M$.

A {\it link map concordance} (also called singular link concordance) between link maps 
$f_0,f_1\:N\to M$ is a link map $H\:N\x I\to M\x I$ such that 
$H(x,i)=(f_i(x),i)$ for $i=0,1$ and $H(x,t)=(f_0(x),t)$ for all $x\in\partial N$.
A link homotopy can be equivalently understood as a level-preserving link map concordance, and we will switch between
the two understandings freely. 
In general link map concordance does not imply imply link homotopy \cite{Sa}, but for link maps of codimension 
at least two it does, by a well-known unpublished result of X.-S. Lin (in the classical dimension) and P. Teichner
(in higher dimensions); see \cite{BT}*{\S5}, \cite{Hab}, \cite{M0}, \cite{M2} and references there 
(see also \cite{HL}*{Theorem 1.7}).

\subsection{(S)LH and (S)LC}
Let $LH_{L,L'}$ be the set of fiberwise link homotopy classes of link homotopies between tangles $L,L'\:N\to M$.
(A homotopy between two homotopies is called {\it fiberwise} if it goes through homotopies.)
Clearly, $LH_{N\to M}:=\bigsqcup_{L,L'\:N\to M} LH_{L,L'}$ is a groupoid, whose composition of morphisms 
is given by $([h],[h'])\mapsto [h*h']$, where $h*h'$ is obtained by consecutive execution of the link homotopies
$h$ and $h'$, and the inverse is given by $[h]\mapsto[\bar h]$, where $\bar h$ obtained by reversing time in $h$.
In particular, $SLH_L:=LH_{L,L}$ is a group under these operations.

Let $LC_{L,L'}$ denote the set of link homotopy classes of link map concordances between tangles $L,L'\:N\to M$.
Similarly to the above, $LC_{N\to M}:=\bigsqcup_{L,L'\:N\to M} LC_{L,L'}$ is a groupoid, and in particular, 
$SLC_L:=LC_{L,L}$ is a group.

\subsection{Action of SLH on LH} \label{action}
$SLH_L$ acts on the left on $LH_{L,L'}$, and this action is transitive and free.
Namely, upon selecting a basepoint $[b]\in LH_{L,L'}$ we get a bijection $\phi_b\:SLH_L\to LH_{L,L'}$ given by 
$[g]\mapsto [g*b]$, whose inverse is given by $[h]\mapsto [h*\bar b]$.
Similarly we also have a bijection $\psi_b\:SLH_{L'}\to LH_{L,L'}$, $[g]\mapsto [b*g]$, 
and the composition $\psi_b^{-1}\phi_b\:SLH_L\to SLH_{L'}$, $[g]\mapsto [\bar b*g*b]$, 
is easily seen to be a group homomorphism, hence an isomorphism.

Now let $\Lambda\:I\sqcup\dots\sqcup I\emb I\x\R^2$ be a string link.
Then $SLH_\Lambda$ is isomorphic to $SLH_\Xi$, and consequently $SLH_\Xi$ has two actions 
on each $LH_{\Lambda,\Lambda'}$ for any choice of $b\in LH_{\Xi,\Lambda}$ and $b'\in LH_{\Xi,\Lambda'}$:
a left action $[g]\cdot[h]=[\bar b*g*b*h]$ and a right action $[h]\cdot[g]=[h*\bar b'*g*b']$.
If $\#$ denotes the usual stacked sum of string links, and also the stacked sum of link homotopies, then
$LH_{\Lambda,\Lambda'}=LH_{\Lambda\#\Xi,\,\Lambda'\#\Xi}$ and $[h]=[h\# 1_\Xi]$, where $1_\Xi$ is the identical 
link homotopy of $\Xi$, and consequently 
\[ [g]\cdot[h]=[(\bar b\#1_\Xi)*(1_\Xi\# g)*(b\#1_\Xi)*(h\#1_\Xi)]=
[(\bar b*1_\Xi*b*h)\#(1_\Xi*g*1_\Xi*1_\Xi)]=[h\# g]\]
and similarly
\[ [h]\cdot[g]=[(h\#1_\Xi)*(\bar b'\#1_\Xi)*(1_\Xi\# g)*(b'\#1_\Xi)]=
[(h*\bar b'*1_\Xi*b')\#(1_\Xi*1_\Xi*g*1_\Xi)]=[h\# g].\]
Thus the two actions are independent of the choice of $b$ and $b'$ and coincide with each other.
In particular, by considering $\Lambda=\Lambda'=\Xi$ we get that $SLH_\Xi$ is abelian, and that its 
multiplication can be described using $\#$ rather than $*$.
However, we do not have an alternative description of the inverse in this group.

\subsection{Inverse in SLC}
All that has been said above about the structure of $LH_{N\to M}$ also applies to $LC_{N\to M}$.
In addition, the inverse in $SLC_\Xi$ can be described using reflection.
Namely, the reflection $\rho\Lambda=\bigsqcup\rho\Lambda_i$ of a string link $\Lambda=\bigsqcup_i\Lambda_i$ is 
defined by precomposing each $\Lambda_i$ with the reflection $I\to I$, $t\mapsto 1-t$, and postcomposing it 
with the reflection $I\x\R^2\to I\x\R^2$, $(t,x)\mapsto (1-t,x)$.
Similarly, the reflection $\rho H=\bigsqcup\rho H_i$ of a link map concordance $H=\bigsqcup_i H_i$ of string links
is defined by precomposing each $H_i$ with the reflection $I\x I\to I\x I$, $(t,s)\mapsto (1-t,s)$, and postcomposing 
it with the reflection $I\x\R^2\x I\to I\x\R^2\x I$, $(t,x,s)\mapsto (1-t,x,s)$.
By a standard construction, $\Lambda\#\rho\Lambda$ is concordant to $\Xi$; similarly, $H\#\rho H$ is
link map concordant to $1_\Xi$.

\section{Algebra: string links}

The present section is essentially an extended version of \cite{MR1}*{\S2.4} (which is two pages of somewhat 
terse text).
Here we slightly improve on the result, present the proof in more detail, and review the necessary background.

\subsection{Kirk--Koschorke invariant}

Let $h=h_+\sqcup h_-\:(I\sqcup I)\x I\to I\x\R^2\x I$ be a generic link homotopy between string links 
$\Lambda,\Lambda'\:I\sqcup I\emb I\x\R^2$; here ``generic'' means that it is a self-transverse regular homotopy with 
at most one double point at each instant of time.
In particular, the sets $\Delta(h_+)$ and $\Delta(h_-)$ of double points of $h_+$ and $h_-$ are finite.

For each $h_+(x,t)=h_+(y,t)=z\in\Delta(h_+)$ let $\eps_z$ be the sign of $z$ (as a normal double point, see
\S\ref{basic2}) and let $J_z=[x,y]$ be the unique arc in $I$ between $x$ and $y$, with orientation induced from 
a fixed orientation of $I$.
Let $l_z$ be the linking number between $h_+(J_z\x\{t\})$ and $h_-(I\x\{t\})$, that is, the class of 
$h_+(J_z\x\{t\})$ in $H_1\big(I\x\R^2\x\{t\}\but h_-(I\x\{t\})\big)\simeq H^1(I,\partial I)\simeq\Z$.

We define $\Sigma(h)$ to be the pair of Laurent polynomials
\[\Sigma_+(h):=\sum_{z\in\Delta(h_+)}\eps_z (t^{l_z}-1)\in\Z[t^{\pm 1}],\]
and $\Sigma_-(h)$, which is defined similarly by interchanging the roles of $h_+$ and $h_-$.
The point of subtracting $1$ is to kill the contribution from each double point $z\in\Delta(h_\pm)$ such that $l_z=0$; 
because of this subtraction, $\Sigma(h)|_{t=1}=(0,0)$.

It is easy to see that $\Sigma(h)$ is invariant under fiberwise link homotopy of $h$, since for each 
integer $g\ne 0$ the oriented $0$-manifold $\Delta_g(h_\pm)=\{z\in\Delta(h_\pm)\mid l_z=g\}$ changes by 
an oriented bordism under a generic fiberwise link homotopy of $h$.
In particular, $\Sigma(h)$ is well-defined for an arbitrary (not necessarily generic) link homotopy $h$ between
string links (since its sufficiently close generic approximations $h'$, $h''$ are related by a generic 
fiberwise link homotopy).

When $h$ is a self-link-homotopy of the trivial string link $\Xi$, $\Sigma(h)$ is essentially
the same as Koschorke's fiberwise, basepoint-preserving version of Kirk's invariant \cite{Ko89}.

Clearly, $\Sigma\:LH_{I\sqcup I\to I\x\R^2}\to\Z[t^{\pm 1}]\oplus\Z[t^{\pm 1}]$ is a morphism of groupoids, 
due to $\Sigma(h*h')=\Sigma(h)+\Sigma(h')$ and $\Sigma(\bar h)=-\Sigma(h)$.
In particular, $\Sigma\:SLH_\Lambda\to\Z[t^{\pm1}]\oplus\Z[t^{\pm1}]$ is a group homomorphism.

\subsection{Jin suspension}

A two-component string link $\Lambda=\Lambda_+\sqcup\Lambda_-\:I\sqcup I\emb I\x\R^2$ 
(where $\Lambda_\pm(i)=(i,\pm1,0)$ for $i\in\partial I$) will be called {\it semi-contractible} if it factors through 
a proper embedding into $I\x [-1,1]\x [-1,1]$, and $\Lambda_\pm$ is homotopic to $\Xi_\pm$ (where 
$\Xi_\pm(t)=(t,\pm 1,0)$ for $t\in I$) with values in $X_\pm:=I\x [-1,1]\x [-1,1]\but\Lambda_\mp(I)$ by a homotopy 
$h_\pm$ keeping the endpoints fixed.
Let us note that any other such homotopy $h_\pm'$ is homotopic to $h_\pm$ through such homotopies since 
$\pi_2(X_\pm)=0$ by the Sphere Theorem of Papakiryakopoulos.
Clearly, if $\Lambda$ is a semi-contractible string link, then $\lk(\Lambda)=0$.
The converse is not true in general, but if $\Lambda$ has unknotted components and zero linking number, then it is 
semi-contractible.
(Indeed, since $\lk(\Lambda)=0$, the loop formed by $\Lambda_\pm$ and $\Xi_\pm$ is null-homologous in $X_\pm$; 
but if $\Lambda_\mp$ is unknotted, then $\pi_1(X_\pm)\simeq\Z$, and therefore this loop is in fact null-homotopic 
in $X_\pm$.) 

Given a semi-contractible string link $\Lambda=\Lambda_+\sqcup\Lambda_-$, a homotopy $h_\pm$ as above extends to 
the link homotopy $H_{\pm t}:=h_{\pm t}\sqcup\Lambda_\mp\:I\sqcup I\to I\times\R^2$ from $\Lambda$ to 
$\Xi_\pm\sqcup\Lambda_\mp$, which moves only one component. 
Let $G_\pm$ be the linear link homotopy from $\Xi_\pm\sqcup\Lambda_\mp$ to $\Xi$, that is, 
$G_{\pm t}=\Xi_\pm\sqcup g_{\pm t}$, where $g_{\pm t}(s)=(1-t)\Lambda_\mp(s)+t\Xi_\mp(s)$.
Then $H_+*G_+$ and $H_-*G_-$ are link homotopies from $\Lambda$ to $\Xi$.
Consequently $\frak J L\:=\bar G_+*\bar H_+* H_-* G_-$ is a self-link-homotopy of $\Xi$, which is well-defined 
up to fiberwise link homotopy and is called the {\it Jin suspension} of the semi-contractible string link $\Lambda$. 
This is parallel to Koschorke's basepoint-preserving version of Jin's construction \cite{Ko89}.
We write $\bar{\frak J}\Lambda$ to mean $\frak J\Lambda$ with time reversed.
Thus $\Sigma(\bar{\frak J}\Lambda)=-\Sigma(\frak J\Lambda)$.

\begin{figure}[h]
\includegraphics{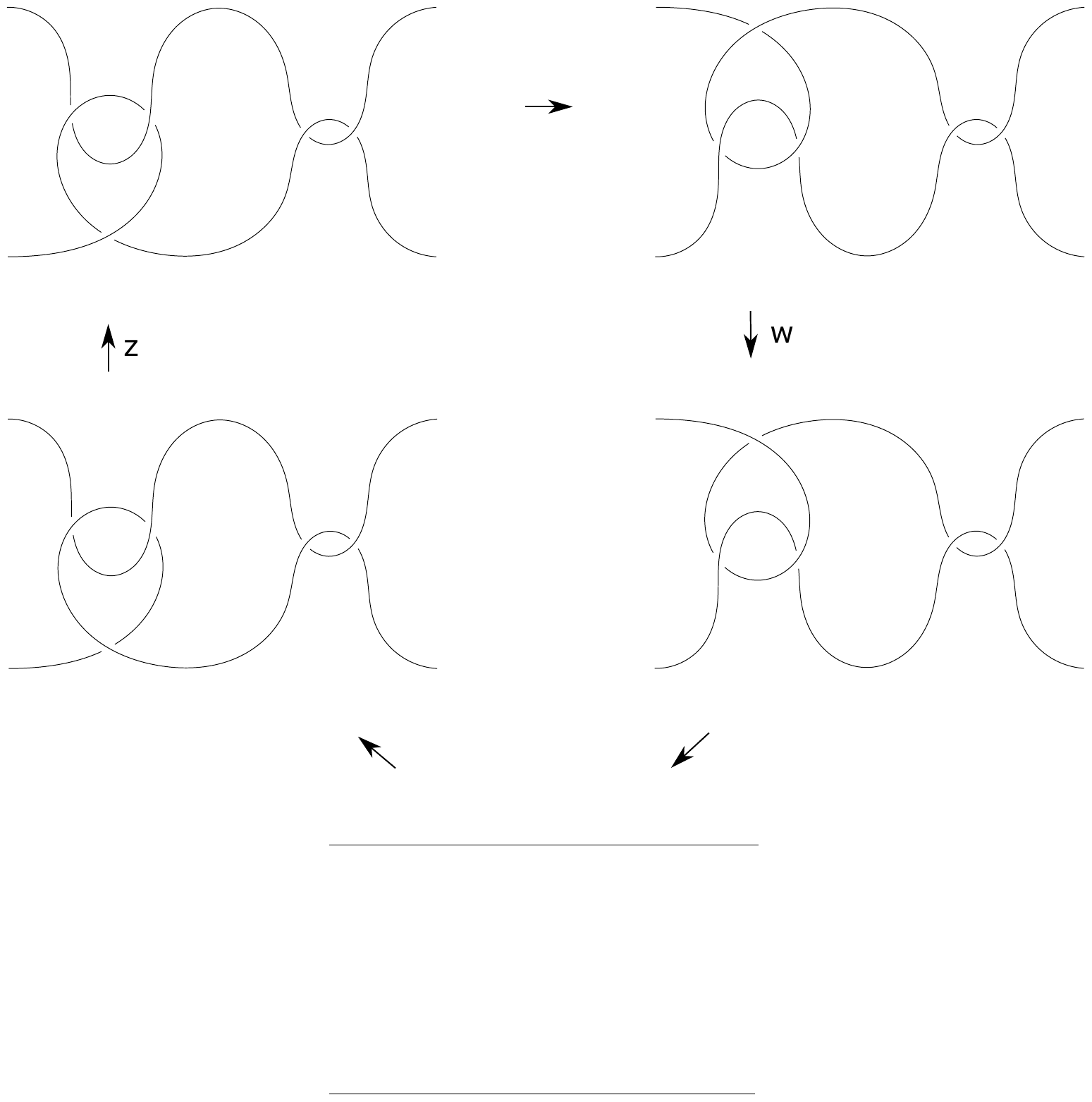}
\caption{The Jin suspension of the Whitehead string link $\Psi$}
\label{fennrolfsen}
\end{figure}

Much of what follows will rest on the following two examples.

\begin{example} (Compare \cite{FR}, \cite{Ko89}, \cite{MR1}.) \label{fenn-rolfsen}
Let $\Psi$ be the Whitehead string link, shown in the top left and (again) in the top right of 
Figure \ref{fennrolfsen}.
If $\cl$ denotes the usual closure operation from string links to links, $\cl\Psi$ is the Whitehead link.
The Jin suspension $\frak J(\Psi)$ is the self-link-homotopy of the trivial string link shown in 
Figure \ref{fennrolfsen}.
It has two double points: a self-intersection $z$ of the first component with $\eps_z=1$ and $l_z=1$; 
and a self-intersection $w$ of the second component with $\eps_w=-1$ and $l_w=1$.
Thus
\[\Sigma(\frak J\Psi)=\big(t-1,\ 1-t\big).\]
\end{example}

\begin{example} (Compare \cite{Ki}, \cite{Ko89}, \cite{MR1}.) \label{Wh-links}
We consider a family of string links $\Psi_{mn}$ such that $\cl\Psi_{mn}$ is the Whitehead link
$(m,1)$-cabled along the first component and $(n,1)$-cabled along the second component.
$\Psi_{3,-2}$ is shown in the top left and (again) in the top right of Figure ~\ref{fennrolfsen2}.
The Jin suspension $\frak J(\Psi_{mn})$ can be represented by a link homotopy with $2m^2+2n^2$ double points
(see Figure \ref{fennrolfsen2} for the case of $\frak J(\Psi_{3,-2})$): 
\begin{itemize}
\item $z_{ij}$ on the 1st component, $1\le i,j\le m$, with 
$\eps_{z_{ij}}=1$ and $l_{z_{ij}}=n\sign(i-j)$;
\item $z_{ij}'$ on the 1st component, $1\le i,j\le m$, with 
$\eps_{z'_{ij}}=-1$ and $l_{z'_{ij}}=0$;
\item $w_{ij}$ on the 2nd component, $1\le i,j\le n$, with 
$\eps_{w_{ij}}=-1$ and $l_{w_{ij}}=m\sign(i-j)$;
\item $w_{ij}'$ on the 2nd component, $1\le i,j\le n$, with 
$\eps_{w'_{ij}}=1$ and $l_{w'_{ij}}=0$.
\end{itemize}
Here $\sign 0=1$. 
Thus
\[\Sigma(\frak J\Psi_{mn})=\big(\tfrac{m^2+m}2(t^n-1)+\tfrac{m^2-m}2(t^{-n}-1),\ 
\tfrac{n^2+n}2(1-t^m)+\tfrac{n^2-n}2(1-t^{-m})\big).\]

For our purposes only $\Psi_n^+:=\Psi_{n1}$ and $\Psi_n^-:=\Psi_{1n}$ will be needed.
We have
\begin{align*}\Sigma(\frak J\Psi_n^+)&=\big(\tfrac{n^2+n}2(t-1)+\tfrac{n^2-n}2(t^{-1}-1),\ 1-t^n\big)\\
&=\big(\tfrac{n^2}2(t+t^{-1}-2)+\tfrac{n}2(t-t^{-1}),\ 1-t^n\big)
\end{align*}
and
\[\Sigma_\pm(\frak J\Psi_n^-)=-\Sigma_\mp(\frak J\Psi_n^+).\]
\end{example}

\begin{example} \cite{MR1} \label{Wh-refl}
Clearly, $\Sigma\big(\frak J(\rho\Lambda)\big)(t)=-\Sigma(\frak J\Lambda)(t^{-1})$ and
$\Sigma\big(\frak J(\Lambda\#\Lambda')\big)=\Sigma(\frak J\Lambda\#\frak J\Lambda'\big)=
\Sigma(\frak J\Lambda)+\Sigma(\frak J\Lambda')$.
Consequently, \[\Sigma\big(\frak J(\Psi_n^+\#\rho \Psi_n^+)\big)=\big(n(t-t^{-1}),\,t^{-n}-t^n\big).\]
Let us note that by a standard construction, $\Psi_n^+\#\rho \Psi_n^+$ is concordant to $\Xi$,
and consequently $\frak J(\Psi_n^+\#\rho \Psi_n^+)$ is fiberwise link map concordant to $1_\Xi$.
Since link map concordance implies link homotopy in codimension two (in the usual, non-fiberwise setting), 
$\frak J(\Psi_n^+\#\rho \Psi_n^+)$ is also link homotopic to $1_\Xi$.
Thus $\Sigma(h)$ is invariant neither under fiberwise link map concordance of $h$, nor under (non-fiberwise)
link homotopy.
\end{example} 

\begin{figure}[h]
\includegraphics[width=15.1cm]{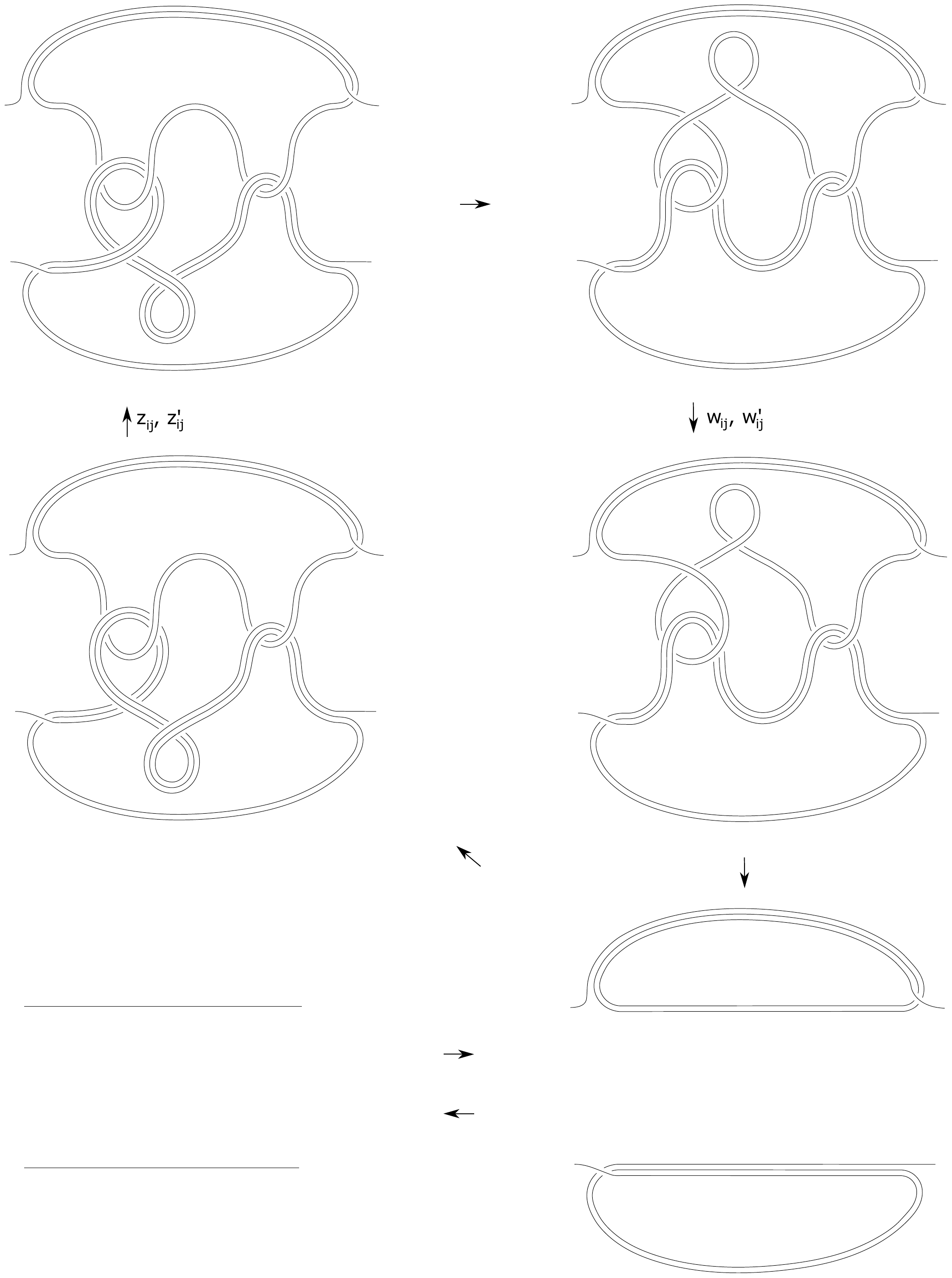}
\caption{The Jin suspension of $\Psi_{3,-2}$}
\label{fennrolfsen2}
\end{figure}

\subsection{Kirk's invariant}

Now let $h=h_+\sqcup h_-\:(I\sqcup I)\x I\to I\x\R^2\x I$ be a generic link map concordance between string links 
$\Lambda,\Lambda'\:I\sqcup I\emb I\x\R^2$; here ``generic'' means that it is a self-transverse smooth immersion.
In particular, the sets $\Delta(h_+)$ and $\Delta(h_-)$ of the double points of $h_+$ and $h_-$ are finite.

For each $h_+(x)=h_+(y)=z\in\Delta(h_+)$ let $J_z\subset I^2$ be an arc between $x$ and $y$.
Then $h_+(J_z)$ is a loop in $I\x\R^2\x I\but h_-(I^2)$.
Let $l_z$ be the linking number between $h_+(J_z)$ and $h_-(I^2)$, that is, the class of $h_+(J_z)$ 
in $H_1\big(I\x\R^2\x I\but h_-(I^2)\big)\simeq H^2(I^2,\partial I^2)\simeq\Z$.
As long as the ordering of the pair of points $(x,y)$ is fixed, $l_z$ does not depend on the choice of $J_z$,
since any two such arcs are homotopic keeping the endpoints fixed.
However, there is no natural ordering of $(x,y)$, so in fact $l_z$ is defined only up to a sign.

Also let $\eps_z=\pm 1$ be the sign of the double point $z$, which is defined to be $+1$ if and only if 
the orientation of $T_z(I\x\R^2\x I)=dh_x(T_x I^2)\oplus dh_y(T_y I^2)$ determined by a fixed orientation of $I^2$ 
agrees with a fixed orientation of $I\x \R^2\x I$.
Namely, these orientations are determined by the usual orientations of $I$ and $I\x\R^2\x I$ 
along with the usual orientation of the time factor $I$.
Let us note that the factor exchanging involution on $\R^2\x\R^2$ is orientation-preserving, so $\eps_z$ 
does not depend on the ordering of $(x,y)$. 

We define $\sigma(h)$ to be the ordered pair of polynomials
\[\sigma_+(h):=\sum_{z\in\Delta(h_+)}\eps_z (t^{|l_z|}-1)\in\Z[t],\]
and $\sigma_-(h)$, which is defined similarly by interchanging the roles of $h_+$ and $h_-$.
It is not hard to see that $\sigma(h)$ is invariant under link homotopy (and hence is well-defined for 
an arbitrary, not necessarily generic, link map concordance $h$ between string links), and even under link map 
concordance of $h$. 
Clearly, $\sigma(h)|_{t=1}=(0,0)$.
If $h$ is a link homotopy between string links, then obviously $\sigma_\pm(h)=|\Sigma_\pm(h)|$, where
$|\cdot|\:\Z[t^{\pm 1}]\to\Z[t]$ is given by $t^n\mapsto t^{|n|}$.

Let $f=f_1\sqcup f_2\:S^2\sqcup S^2\to S^4$ be a link map.
Let us connect a generic point in $f_1(S^2)$ with a generic point in $f_2(S^2)$ by an arc $J$ that is otherwise 
disjoint from $f(S^2\sqcup S^2)$.
Then an appropriate small regular neighborhood of $J$ in $S^4$ is a $4$-ball $B$ that meets each $f_i(S^2)$ in 
a $2$-disk.
Its exterior $\Cl{S^4\but B}$ is homeomorphic to $I\x[-1,1]\x[0,3]\x I$ by a homeomorphism $h$ that takes
each $f_i(S^2)\cap\partial B$ onto $I\x\{0\}\x\{i\}\x\partial I\cup(\partial I)\x\{0\}\x\{i\}\x I$.
Thus $hf$ is a self-link-map-concordance of $\Xi$.
Then $\sigma(f):=\sigma(hf)$ is nothing but Kirk's original invariant of $f$ \cite{Ki} (see also 
\cite{Ki2}, \cite{Ko90} for direct proofs of the invariance of Kirk's invariant under link map concordance).
It has long been an open problem whether Kirk's invariant of link maps $S^2\sqcup S^2\to S^4$ is injective; 
according to the recent preprint \cite{ST}, it is.

Clearly, $\sigma\:LC_{I\sqcup I\to I\x\R^2}\to\Z[t]\oplus\Z[t]$ is a morphism of groupoids, and in particular
$\sigma\:SLC_\Lambda\to\Z[t]\oplus\Z[t]$ is a group homomorphism.

\subsection{Generalized Sato--Levine invariant}\label{genSL}

The generalized Sato--Levine invariant $\beta(L)$ of the two-component link $L=K\sqcup K'\:S^1\sqcup S^1\emb S^3$ 
was discovered independently by Akhmetiev (see \cite{AR}), Kirk--Livingston \cite{KL}, Polyak--Viro and Traldi
(see references in \cite{MR1}*{\S2.2, Remark (i)}).
In fact, as observed in \cite{KMT}*{proof of Theorem 7.2}, it can be defined as 
$\beta(L)=c_3(L)+c_1(L)\big(c_2(K)+c_2(K')\big)$, where $c_i$ is the coefficient of the Conway polynomial at $z^i$.
From this it is easy to check (see \cite{Liv}*{Theorem 8.7} or \cite{Ni}*{Lemma 2.1} for the details) that 
for a pair of two-component links $L_+=K_+\sqcup K'$ 
and $L_-=K_-\sqcup K'$ related by a single crossing change on the first component,
\[\beta(L_+)-\beta(L_-)=\eps\lk(K_1,K')\lk(K_2,K'),\]
where $K_1$ and $K_2$ are the two lobes of the intermediate singular knot, and $\eps$ is the sign of the crossing
change (see \S\ref{basic2} concerning this sign).
The similar formula holds for crossing changes on the second component.
From these crossing change formulas and Figure \ref{simdelta}
it is easy to see that $\beta$ is an invariant of self $C_2$-equivalence.

\begin{example} \label{triviality}
Let $\Lambda_k$ be some fixed string link with $\lk(\Lambda_k)=k$ and with $\beta(\cl\Lambda_k)=0$; 
for example, we may take $\Lambda_k$ to be the pure braid $\sigma_{12}^{2k}$.
Let us write $\cl_k\Lambda=\cl(\Lambda\#\Lambda_k)$.
Thus $\beta(\cl_k\Xi)=0$ for all $k$.

Let $\Psi$ be the Whitehead string link (see Example \ref{fenn-rolfsen}).
There is an obvious link homotopy between $\Psi\#\rho\Psi$ and $\Xi$ with two double points in the first component.
These double points have opposite signs, and also opposite linking numbers of the loops $l_z$ with
the second component.
Hence \[\beta\big(\cl_0(\Psi\#\rho\Psi)\big)-\beta(\cl_0\Xi)=1\cdot 1\cdot (-1)+(-1)\cdot (-1)\cdot 1=0,\] and 
consequently $\beta\big(\cl_0(\Psi\#\rho\Psi)\big)=0$.
On the other hand, \[\beta\big(\cl_1(\Psi\#\rho\Psi)\big)-\beta(\cl_1\Xi)=1\cdot 1\cdot 0+(-1)\cdot(-1)\cdot 2=0,\] 
and consequently $\beta\big(\cl_0(\Psi\#\rho\Psi)\big)=2$.

Thus $\Psi\#\rho\Psi$ is not self $C_2$-equivalent to $\Xi$.
On the other hand, $\Psi\#\rho\Psi$ is concordant to $\Xi$ by a standard construction, and in particular all
its $\mu$-invariants vanish (one of them is $\mu(1122)=\beta\circ\cl_0$).
Thus self $C_2$-triviality of string links is not detected by their $\mu$-invariants.
\end{example}

\subsection{Constraints}

Let $h=h_+\sqcup h_-$ be a generic link homotopy between the string links $\Lambda,\Lambda'\:I\sqcup I\emb I\x\R^2$.
Let $\lambda=\lk(\Lambda)=\lk(\Lambda')$.
Then $\lk(\cl_{k-\lambda}L)=k$, so that, writing $\beta_k(\Lambda)=\beta(\cl_{k-\lambda}\Lambda)$, we have
\[\beta_k(\Lambda')-\beta_k(\Lambda)=
\sum_{z\in\Delta(h_+)}\eps_z l_z (k-l_z)+\sum_{z\in\Delta(h_-)}\eps_z l_z (k-l_z).\]
This yields two constraints on $\Sigma$:
\[\tfrac{d}{dt}\Sigma_+(h)+\tfrac{d}{dt}\Sigma_-(h)\Big|_{t=1}=
\beta_1(\Lambda')-\beta_1(\Lambda)+
\beta_0(\Lambda)-\beta_0(\Lambda')\]
and
\[\tfrac{d^2}{dt}\Sigma_+(h)+\tfrac{d^2}{dt}\Sigma_-(h)\Big|_{t=1}=
\beta_1(\Lambda)-\beta_1(\Lambda'),\]
and one constraint on $\sigma$:
\[\tfrac{d}{dt}\sigma_+(h)+\tfrac{d}{dt}\sigma_-(h)+
 \tfrac{d^2}{dt}\sigma_+(h)+\tfrac{d^2}{dt}\sigma_-(h)\Big|_{t=1}=
\beta_0(\Lambda)-\beta_0(\Lambda').\]

It turns out that these constraints suffice to determine the images of $\Sigma$ and $\sigma$:

\begin{theorem} \label{image1}
Let $\Lambda$ and $\Lambda'$ be link homotopic two-component string links.

(a) $\Sigma(LH_{\Lambda,\Lambda'})$ equals
\[\Delta^{-1}\big(0,0,\ \beta_1(\Lambda')-\beta_1(\Lambda)+\beta_0(\Lambda)-\beta_0(\Lambda'),\ \,
\beta_1(\Lambda)-\beta_1(\Lambda')\big),\]
where $\Delta\:\Z[t^{\pm 1}]\oplus\Z[t^{\pm1}]\to\Z\oplus\Z\oplus\Z\oplus\Z$ is the group homomorphism given by 
\[(f,g)\mapsto \big(f|_{t=1},\,g|_{t=1},\,f'+g'|_{t=1},\,f''+g''|_{t=1}\big).\]

(b) $\sigma(LH_{\Lambda,\Lambda'})$ equals
\[\delta^{-1}\big(0,0,\,\beta_0(\Lambda)-\beta_0(\Lambda')\big),\] 
where $\delta\:\Z[t]\oplus\Z[t]\to\Z\oplus\Z\oplus\Z$ is the group homomorphism given by 
\[(f,g)\mapsto \big(f|_{t=1},\,g|_{t=1},\,f'+g'+f''+g''|_{t=1}\big).\]
\end{theorem}

\begin{proof} $\Sigma(LH_{\Lambda,\Lambda'})\subset\Delta_{\Lambda,\Lambda'}:=
\Delta^{-1}\big(0,0,\,\beta_1(\Lambda')-\beta_1(\Lambda)+\beta_0(\Lambda)-\beta_0(\Lambda'),\,
\beta_1(\Lambda)-\beta_1(\Lambda')\big)$ and $\sigma(LH_{\Lambda,\Lambda'})\subset\delta_{\Lambda,\Lambda'}:=
\delta^{-1}\big(0,0,\,\beta_0(\Lambda)-\beta_0(\Lambda')\big)$ 
by the above constraints.
It remains to prove the reverse inclusions.
We first consider the case $\Lambda=\Lambda'=\Xi$.

Since $\sigma_\pm(h)=|\Sigma_\pm(h)|$, from Example \ref{Wh-links} we have 
$\sigma(\frak J\Psi_n^+)=\big(n^2(t-1),\,1-t^n\big)$ and
$\sigma(\frak J\Psi_n^-)=\big(t^n-1,\,n^2(1-t)\big)$. 
But these clearly generate $\ker\delta$.
(Let us note that $\ker\delta$ consists of pairs $(a_mt^m+\dots+a_0,\, b_nt^n+\dots+b_0)$ where the $a_i$ with $i\ge 2$ 
and the $b_i$ with $i\ge 1$ are arbitrary, and $a_0$, $b_0$ and $a_1$ are fully determined.)
Thus $\sigma(SLH_\Xi)\supset\ker\delta$.

It is easy to see that the composition $\Z[t^{\pm1}]^2\xr{\Delta}\Z^4\xr{P}\Z^3$, where $P(a,b,c,d)=(a,b,c+d)$, coincides 
with $\Z[t^{\pm1}]^2\xr{Q}\Z[t]^2\xr{\delta}\Z^3$, where $Q(f,g)=(|f|,|g|)$.
Consequently, if $(p_+,p_-)\in\ker\Delta$, then $(|p_+|,|p_-|)\in\ker\delta$,
so there exists a self-link-homotopy $h$ of 
$\Xi$ (a suitable combination of the $\frak J\Psi_n^\pm$) such that $|p_\pm|=\sigma_\pm(h)=|\Sigma_\pm(h)|$.
Then $\big(p_+-\Sigma_+(h),p_--\Sigma_-(h)\big)\in\ker\psi$, where
$\psi\:\Z[t^{\pm 1}]\oplus\Z[t^{\pm 1}]\to\Z[t]\oplus\Z[t]\oplus\Z$ is given by $(f,g)\mapsto (|f|,|g|,f'+g'|_{t=1})$.
But it is easy to see that $\ker\psi$ is generated by the values of $\Sigma$ from Example \ref{Wh-refl},
$\Sigma\big(\frak J(\Psi_n^+\#\rho \Psi_n^+)\big)=\big(n(t-t^{-1}),\,t^{-n}-t^n\big)$ and
$\Sigma\big(\frak J(\Psi_n^-\#\rho \Psi_n^-)\big)=\big(t^n-t^{-n},\,n(t^{-1}-t)\big)$.
Thus $\Sigma(SLH_\Xi)\supset\ker\Delta$.

Finally, given any link homotopy $h$ between $\Lambda$ and $\Lambda'$ and any $x\in\Delta_{\Lambda,\Lambda'}$,
we have $x-\Sigma(h)\in\Delta_{\Xi,\Xi}=\ker\Delta$.
By the above, $x-\Sigma(h)=\Sigma(h_\Xi)$ for some self-link-homotopy $h_\Xi$ of $\Xi$.
Now the stacked sum $h\# h_\Xi$ is a link homotopy between $\Lambda$ and $\Lambda'$ such that 
$\Sigma(h\# h_\Xi)=\Sigma(h)+\Sigma(h_\Xi)=x$.
Thus $\Delta_{\Lambda,\Lambda'}\subset\Sigma(LH_{\Lambda,\Lambda'})$. 
The inclusion $\delta_{\Lambda,\Lambda'}\subset\sigma(LH_{\Lambda,\Lambda'})$ is proved similarly.
\end{proof}

Since link map concordances are link homotopic to link homotopies in codimension two and $\sigma$ is invariant under 
link homotopy, $\sigma(LH_{\Lambda,\Lambda'})=\sigma(LC_{\Lambda,\Lambda'})$.
Also it is easy to see that $\sigma(LC_{\Lambda,\Lambda'})=\sigma(LM_{S^2\sqcup S^2\to S^4})$.

\begin{corollary}{\rm (Kirk \cite{Ki})} \label{image-kirk} $\sigma(LM_{S^2\sqcup S^2\to S^4})=\ker\delta$.
\end{corollary}

\begin{corollary} {\rm ($\supset$: Koschorke \cite{Ko89}, $\subset$: Melikhov--Repov\v s \cite{MR1})} 
\label{koschorke+} $\Sigma(SLH_\Xi)=\ker\Delta$. 
\end{corollary}

\begin{corollary} \label{realization1}
If $\Lambda$ and $\Lambda'$ are link homotopic 2-component string links and $\beta_i(\Lambda)=\beta_i(\Lambda')$ 
for $i=0,1$, then $\Lambda$ and $\Lambda'$ are joined by a link homotopy $h$ such that $\Sigma(h)=0$.
\end{corollary}

\section{Algebra: links}

\subsection{Kirk--Livingston invariant}
Let $H=H_+\sqcup H_-\:S^1\x I\sqcup S^1\x I\to S^3\x I$ be a generic link homotopy between links 
$L,L'\:S^1\sqcup S^1\emb S^3$.
Similarly to the above we may define its $\Sigma$-invariant, the only difference being that for each 
$z=H_+(x,t)=H_+(y,t)$ in the double point set $\Delta(H_+)$ of $H_+$ there are now two ways of choosing an arc 
$J_z\subset S^1$ between $x$ and $y$, which are equally natural and distinct up to homotopy.
However, the orientation of $J_z$ can still be induced from a fixed orientation of $S^1$.
This leads to two possible values $l_z$ and $l'_z$ of the linking number between $H_+(J_z\x\{t\})$ and $h_-(S^1\x\{t\})$
in $\R^3\x\{t\}$.
Their sum is $\lambda:=\lk(L)=\lk(L')$.
In particular, when $\lambda=0$ we have $|l_z|=|l'_z|$ and it is natural to define 
\[\Sigma_+(H)=\sum_{z\in\Delta(H_+)}\eps_z (t^{|l_z|}-1)\in\Z[t].\]
In general, we have $|l_z-\frac\lambda2|=|l_z'-\frac\lambda2|$, and we define 
$\Sigma(H)=\big(\Sigma_+(H),\Sigma_-(H)\big)$, where
\[\Sigma_+(H)=\sum_{z\in\Delta(H_+)}\eps_z (t^{|l_z-\frac\lambda2|-|\frac\lambda2|}-1)\in 
t^{-\lfloor|\frac\lambda2|\rfloor}\Z[t],\]
and $\Sigma_-(H)$ is defined similarly, by interchanging the roles of $H_+$ and $H_-$.
The point of subtracting $\frac{|\lambda|}2$ is to make the exponents integer in all cases (rather than
half-integer) and to keep killing the free term, rather than some other coefficient.
We have 
\[\left|x-\tfrac\lambda2\right|-\left|\tfrac\lambda2\right|=\begin{cases}
x&\text{ if }\lambda\le 0\text{ and }x\ge\lambda/2;\\
\lambda-x&\text{ if }\lambda\le 0\text{ and }x\le\lambda/2;\\
x-\lambda&\text{ if }\lambda\ge 0\text{ and }x\ge\lambda/2;\\
-x&\text{ if }\lambda\ge 0\text{ and }x\le\lambda/2.
\end{cases}\]
Let us also note that $|l_z-\frac\lambda2|-|\frac\lambda2|=\max(\epsilon l_z,\epsilon l'_z)$,
where $\epsilon=-\sign \lambda$ if $\lambda\ne 0$ and $\epsilon=\pm 1$ (chosen arbitrarily) if $\lambda=0$.
Clearly, $\Sigma(H)|_{t=1}=(0,0)$, and $\Sigma(H)$ is invariant under fiberwise link 
homotopy of $H$, and consequently is well-defined for an arbitrary (not necessarily generic) link homotopy $H$ 
between links.

\begin{remark}
Let us note that if $H$ were assumed to be only a link map concordance, then for each double point 
$H_+(x,t)=H_+(y,s)=z$ there would have been $\pi_1(S^1\x I)\simeq\Z$ of equally natural ways of choosing 
an arc $J\subset S^1\x I$ between $(x,t)$ and $(y,s)$, as long as the ordering of $(x,y)$ is fixed.
In this case only the ``absolute value'' of the $\bmod\lambda$ residue class of $l_z$ is well-defined, as 
an element of $\{0,\dots,\lfloor\lambda/2\rfloor\}$ if $\lambda\ne 0$ or as a non-negative integer if $\lambda=0$.
\end{remark}

\begin{example} \label{Wh-twists} 
As before, we fix some string link $\Lambda_k$ with $\lk(\Lambda_k)=k$ (but now not necessarily
with $\beta(\cl\Lambda_k)=0$) and write $\cl_k\Lambda=\cl(\Lambda\#\Lambda_k)$.

The operation $\cl_k$ can also be applied to a link homotopy $h$ of string links in the obvious way.
Clearly, $\Sigma_\pm(\cl_k h)=|\Sigma_\pm(h)|_k$, where
$|\cdot|_0\:\Z[t^{\pm 1}]\to\Z[t]$ (previously denoted by $|\cdot|$) is given by $t^n\mapsto t^{|n|}$, 
and in general $|\cdot|_k\:\Z[t^{\pm 1}]\to t^{-\lfloor|\frac k2|\rfloor}\Z[t]$ is given by
$t^n\mapsto t^{|n-\frac k2|-|\frac k2|}$.

Thus we have

\[\Sigma(\cl_k\frak J\Psi_n^+)=
\begin{cases}
\big(\tfrac{n^2}2(t+t^{-1}-2)+\tfrac{n}2(t-t^{-1}),\,1-t^n\big)& \text{ if }k\le-2\text{ and }n\ge\frac k2\\
\big(\tfrac{n^2+n}2(t-1),\,1-t^n\big)& \text{ if }k=-1\text{ and }n\ge 0\\
\big(n^2(t-1),\,1-t^n\big)& \text{ if }k=0\\
\big(\tfrac{n^2-n}2(t-1),\,1-t^{-n}\big)& \text{ if }k=1\text{ and }n\le 0\\
\big(\tfrac{n^2}2(t+t^{-1}-2)-\tfrac{n}2(t-t^{-1}),\,1-t^{-n}\big)& \text{ if }k\ge2\text{ and }n\le\frac k2\\
\end{cases}\]

Also,
\[\Sigma(\cl_k\frak J\Psi_{k-n}^+)=
\begin{cases}
\big(\tfrac{(k-n)^2}2(t+t^{-1}-2)+\tfrac{k-n}2(t-t^{-1}),\,1-t^n\big)& \text{ if }k\le-2\text{ and }n\ge\frac k2\\
\big(\tfrac{(k-n)^2}2(t+t^{-1}-2)-\tfrac{k-n}2(t-t^{-1}),\,1-t^{-n}\big)& \text{ if }k\ge2\text{ and }n\le\frac k2.\\
\end{cases}\]
\end{example}

\begin{example} 
Given two link homotopies $h$, $h'$ between string links, $\cl_k(h\# h')$ does not seem to be easily
obtainable from $\cl_k h$ and $\cl_k h'$.
Nevertheless, $\Sigma_\pm\big(\cl_k(h\# h')\big)=|\Sigma_\pm(h\#h')|_k=
|\Sigma_\pm(h)|_k+|\Sigma_\pm(h')|_k=\Sigma_\pm(\cl_k h)+\Sigma_\pm(\cl_k h')$.

In particular, if either $k\le-2$ and $n\ge\frac k2$ or $k\ge2$ and $n\le\frac k2$, then
\[\Sigma\big(\cl_k(\bar{\frak J}\Psi_n^+\#\frak J\Psi_{k-n}^+)\big)=
\big(\tfrac{k(k-2n)}2(t+t^{-1}-2)+\epsilon\tfrac{k-2n}2(t-t^{-1}),\ 0\big),\]
where $\epsilon=-\sign k$.

If $k$ is odd, $k\ne\pm 1$, then in the case $k=2n-\epsilon$ we get 
\[\Sigma\big(\cl_k(\bar{\frak J}\Psi_{(k-1)/2}^+\#\frak J\Psi_{(k+1)/2}^+)\big)=
\big(\tfrac k2(t+t^{-1}-2)+\tfrac{\epsilon}2(t-t^{-1}),\ 0\big).\]
If $k$ is even, $k\ne 0$, then in the case $k=2n-2\epsilon$ we get
\[\Sigma\big(\cl_k(\bar{\frak J}\Psi_{(k-2)/2}^+\#\frak J\Psi_{(k+2)/2}^+)\big)=
\big(k(t+t^{-1}-2)+\epsilon(t-t^{-1}),\ 0\big).\]
\end{example}

\begin{example}
If $k$ is odd, $k\ne\pm 1$, let $h_n$ denote the stacked sum of $n$ copies of 
$\frak J\Psi_{(k-1)/2}^+\#\bar{\frak J}\Psi_{(k+1)/2}^+$.
Then, if either $k\le-2$ and $n\ge\frac k2$ or $k\ge2$ and $n\le\frac k2$,
\[\Sigma\big(\cl_k(\frak J\Psi_n^+\# h_n)\big)=
\big(\tfrac{n(n-k)}2(t+t^{-1}-2),\,1-t^{\epsilon n}\big).\]

If $k$ is even, $k\ne 0$, let $h_m$ be the stacked sum of $m$ copies of 
$\frak J\Psi_{(k-2)/2}^+\#\bar{\frak J}\Psi_{(k+2)/2}^+$.
Then, if either $k\le-2$ and $n\ge\frac k2$ or $k\ge2$ and $n\le\frac k2$
\[\Sigma\big(\cl_k(\frak J\Psi_n^+\# h_{(\epsilon n(k-n)+n)/2})\big)=
\Big(n(n-k)\big(t-1-\tfrac{|k|}2(t+t^{-1}-2)\big),\,1-t^{\epsilon n}\Big).\]
\end{example}

\subsection{Constraints}

For a generic link homotopy $H=H_+\sqcup H_-$ between links $L,L'\:S^1\sqcup S^1\emb S^3$,
\[\beta(L')-\beta(L)=\sum_{z\in\Delta(H_+)}\eps_z l_z l'_z+\sum_{z\in\Delta(H_-)}\eps_z l_z l'_z.\]
Taking into account that $l_zl'_z=l_z(\lambda-l_z)=-\epsilon l_z(\epsilon l_z+|\lambda|)$, where $l_z$ and $l'_z$
are interchangeable, this yields a constraint on $\Sigma$:
\[\tfrac{d}{dt}\big(t^{1+|\lambda|}\tfrac{d}{dt}\Sigma_+(H)\big)+
\tfrac{d}{dt}\big(t^{1+|\lambda|}\tfrac{d}{dt}\Sigma_-(H)\big)=\beta(L)-\beta(L').\]

It turns out that this constraint suffices to determine the image of $\Sigma$:

\begin{theorem} \label{image2} Let $L$ and $L'$ be two-component links with $\lk(L)=\lk(L')=\lambda$.
Then
\[\Sigma(LH_{L,L'})=\Delta_\lambda^{-1}\big(0,0,\,\beta(L)-\beta(L')\big),\] 
where $\Delta_\lambda\:
t^{-\lfloor|\frac\lambda2|\rfloor}\Z[t]\oplus t^{-\lfloor|\frac\lambda2|\rfloor}\Z[t]\to\Z\oplus\Z\oplus\Z$ 
is the group homomorphism given by 
\[(f,g)\mapsto \big(f|_{t=1},g|_{t=1},\,(1+|\lambda|)(f'+g')+f''+g''|_{t=1}\big).\]
\end{theorem}

\begin{proof}
$\Sigma(LH_{L,L'})$ is contained in the specified coset $\Delta_{L,L'}$ of $\ker\Delta_\lambda$ by the above constraint.
It remains to prove the reverse inclusion.

If $\lambda$ is odd, then $(1+|\lambda|)f'+f''|_{t=1}$ is even for every
$f\in t^{-\lfloor|\frac\lambda2|\rfloor}\Z[t]$, and there exists an 
$f\in t^{-\lfloor|\frac\lambda2|\rfloor}\Z[t]$ such that $(1+|\lambda|)f'+f''|_{t=1}=2$, namely,
$f(t)=t$ if $\lambda=\pm 1$ and $f(t)=t+t^{-1}$ if $\lambda\ne\pm 1$. 
If $\lambda$ is even, then there exists an $f\in t^{-\lfloor|\frac\lambda2|\rfloor}\Z[t]$ such that 
$(1+|\lambda|)f'+f''|_{t=1}=1$, namely, $f(t)=t$ if $\lambda=0$ and $f(t)=t-1+\frac{|\lambda|}2(t+t^{-1}-2)$
if $\lambda\ne 0$.
Using this observation, it is easy to see that the values of $\Sigma$ listed in Example \ref{Wh-twists} 
generate $\ker\Delta_\lambda$.
Thus $\Sigma(SLH_\Xi)\supset\ker\Delta_\lambda$.

Now given any link homotopy $H$ between $L$ and $L'$ and any $x\in\Delta_{L,L'}$, 
we have $x-\Sigma(H)\in\ker\Delta_\lambda$.
Hence by the above there exists a self-link-homotopy $h_\Xi$ of $\Xi$
such that $\Sigma(\cl_\lambda h_\Xi)=x-\Sigma(H)$.
Here $\cl_\lambda$ has been defined using an arbitrary string link $\Lambda_\lambda$ of linking number $\lambda$.
In particular, we may choose $\Lambda_\lambda$ so that $L=\cl(\Lambda_\lambda)$.
Then $\cl_\lambda(h_\Xi)$ is a self-link-homotopy of $L$.
By combining it with $H$ we obtain a link homotopy $H'$ from $L$ to $L'$ such that 
$\Sigma(H')=\Sigma(H)+\Sigma(\cl_\lambda h_\Xi)=x$.
\end{proof}

\begin{corollary} \label{kltheorem} \cite{KL}
$\Sigma(SLH_L)=\ker\Delta_{\lk(L)}$.
\end{corollary}

This result is stated at the end of the long Remark in \cite{KL}*{\S4}, which is apparently due to the anonymous 
referee of \cite{KL}.
However, the proof of this result is omitted (``In fact, a simple algebraic argument as in the proof of Lemma 4.2 shows 
that the image of $\Phi$ equals the kernel of $\Psi$'').
In this connection let us note that the above proof of Corollary \ref{kltheorem} can well 
be called ``simple'' from the conceptual viewpoint --- but perhaps not from the technical one.

\begin{corollary} \label{realization2}
If $\lk(L)=\lk(L')$ and $\beta(L)=\beta(L)$, then $L$ and $L'$ are
joined by a link homotopy $H$ such that $\Sigma(H)=0$.
\end{corollary}

\section{Geometry} \label{geometry}

\subsection{Whitney disks for singular links} 

A singular tangle $f\:N\to M$ will be called a {\it $(-\frac12)$-quasi-tangle} if it has precisely two double points
$z=f(x)=f(y)$ and $z'=f(x')=f(y')$ and there exist disjoint arcs $J_+,J_-\subset N$ such that 
$f(\partial J_+)=f(\partial J_-)=\{z,z'\}$ and the loop given by the concatenation of paths $f|_{J_+}$ and $f|_{J_-}$ 
is null-homotopic in $M$.
Writing $B(x,y)$ for the unit ball in $\R^2$ centered at $(x,y)$, such a null-homotopy can be thought of as a generic 
smooth map $W$ of the two-cornered smooth disk $D:=B(0,0)\cap B(1,1)$ into $M$ sending the lower half-boundary 
$\partial_-D:=D\cap\partial B(0,0)$ homeomorphically onto $f(J_-)$ and the upper half-boundary 
$\partial_+ D:=D\cap\partial B(1,1)$ homeomorphically onto $f(J_+)$.
Since $D$ is orientable, pairs of Whitney umbrellas in $W$ can be cancelled along double point curves, and thus 
$W$ can be chosen to be a smooth immersion.
Such an immersion will be called a {\it Whitney disk} pairing up the double points $z$, $z'$ of 
the $(-\frac12)$-quasi-tangle $f$.

A singular tangle $f\:N\to M$ will be called a {\it $\frac12$-quasi-tangle} if $N=K\sqcup L$, where $K$ is connected,
$f|_L$ is an embedding, and $f|_K\:K\to M\but f(L)$ is a $(-\frac12)$-quasi-tangle.

By a {\it normal homotopy} we mean a proper smooth homotopy of a compact 1-manifold $N$ in a 3-manifold that keeps 
$\partial N$ fixed and has finitely many double points, all of which are normal (see \S\ref{basic2}).
Let us note that it may have double points occurring at the same moment.

A {\it $\pm\frac12$-quasi-isotopy} is a normal homotopy $h_t\:N\to M$ whose every 
instant $h_s$ is either a tangle or a $\pm\frac12$-quasi-tangle whose two double points have opposite signs 
(as normal double points of the homotopy).
The terminology is motivated by \cite{MR1}, where $n$-quasi-isotopy is defined for $n=0,1,2,\dots$.

\begin{lemma} \label{whitney1}
Two tangles are $(-\frac12)$-quasi-isotopic if and only if they are $C_2$-equivalent.
\end{lemma}

By a theorem of Matveev \cite{Matv} and Murakami--Nakanishi \cite{MN}, two links are $C_2$-equivalent
if and only if they have the same linking numbers of the corresponding $2$-component sublinks.
However, $C_2$-equivalence for tangles in link complements is more interesting.

\begin{proof} Let $J_1$, $J_2$ and $J_3$ be the $3$ strands of the given $C_2$-move (in any order).
The $C_2$-move has the same effect as a certain homotopy whose only intersection points are two normal double points, 
each between $J_1$ and $J_2$, which occur simultaneously in time and have opposite signs (see Figure \ref{simdelta}).
These two double points can be paired up by a small embedded Whitney disk, which transversally intersects $J_3$ in 
one point.

\begin{figure}[h]
\includegraphics[width=10cm]{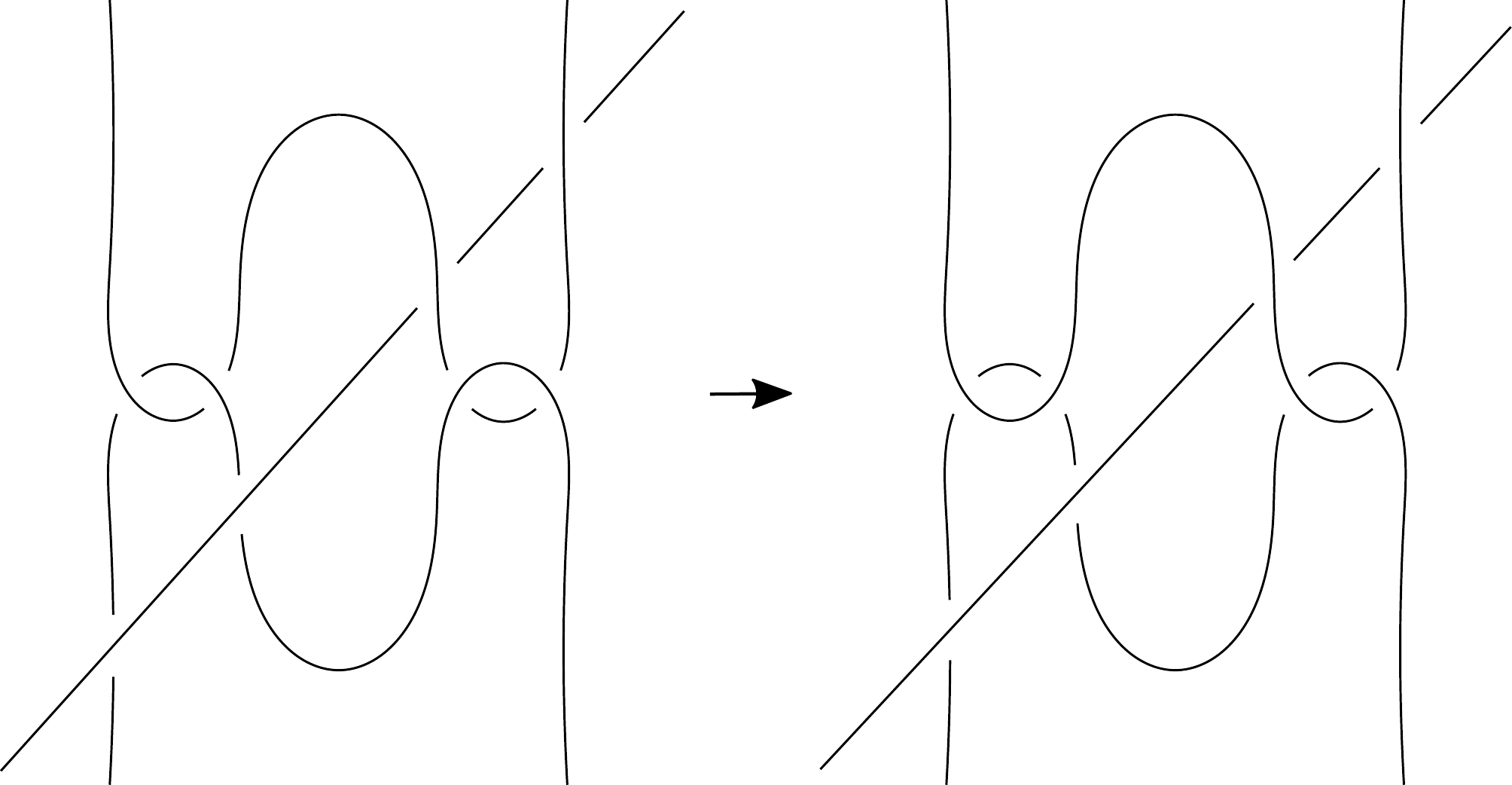}
\caption{$C_2$-move realized by two simultaneous $C_1$-moves}
\label{simdelta}
\end{figure}

Conversely, let $h_t\:N\to M$ be a $(-\frac12)$-quasi-isotopy and $h_s$ be its singular instant.
Thus $h_s$ is a $(-\frac12)$-quasi-tangle whose two double points $p$, $q$ have opposite signs (as normal double points 
of $h_t$) and are paired up by an immersed Whitney disk $W\:D\to M$, transverse to $h_s(N)$ except at $W(\partial D)$.
Thus the immersed Whitney disk meets the image of $N$ transversally in a finite number $n$ of interior points.
We want to replace $W$ by $n$ embedded Whitney disks, each meeting the image of $N$ transversally in just one 
interior point --- just like the Whitney disk arising as above from a $C_2$-move.

To do so, let us represent $D$ as the image of $I\x I$ under a quotient map $Q$ whose only non-singleton point-inverses
are $\{0\}\x I$ and $\{1\}\x I$, sent by $Q$ to the corners of $D$.
We may assume by general position that for $t\ne 0,1$ each segment $\{t\}\x I$ is embedded by $WQ$ and its image meets 
$h_s(N)$ in at most one interior point.
Pick a sequence of points $0=t_0<\dots<t_n=1$ such that each strip $[t_i,t_{i+1}]\x I$ is embedded by $WQ$ and its image
meets $h_s(N)$ in a single interior point.
It can be arranged that for a sufficiently small $\eps>0$ the interval $[s-\eps,s+\eps]$ of the homotopy $h_t$ 
has support in a neighborhood $U$ of $h_s^{-1}(\{p,q\})$ whose image is disjoint from $WQ([t_1,t_{n-1}]\x I))$.

\begin{figure}[h]
\includegraphics{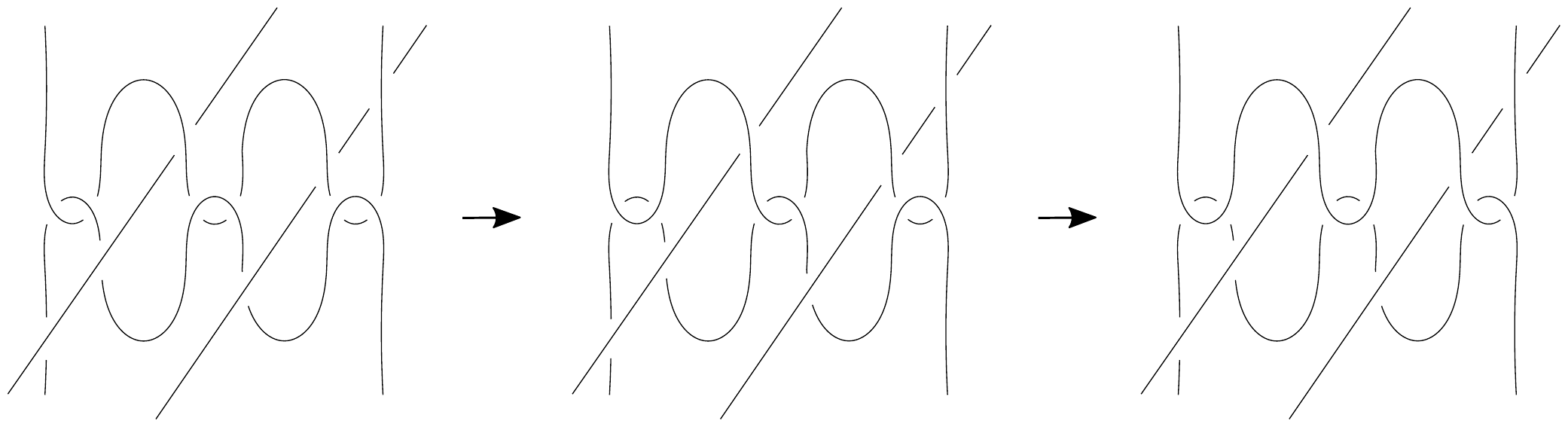}
\caption{$(-\frac12)$-quasi-tangle replaced by $n=2$ pairs of simultaneous $C_1$-moves}
\label{simdelta2}
\end{figure}

We join the tangles $L_0:=h_{s-\eps}$ and $L_{2n}:=h_{s+\eps}$ by the following sequence of $2n$ normal crossings 
(i.e., a normal homotopy).
See Figure \ref{simdelta2} for the case $n=2$.
Start by performing the original crossing of $h_t$ at $p$ (without making any crossing at $q$ yet); by symmetry 
we may assume that it is a positive crossing.
Let $L_1$ denote the resulting tangle.
Now cross the two strands $WQ([\frac{t_0+t_1}2,\frac{t_1+t_2}2]\x\partial I)$ by pushing each halfway to the other one 
along the segment $W(t_1\x I)$ and introducing a half-twist so as to get a negative crossing.
Let $L_2$ denote the resulting tangle.
Then undo the homotopy from $L_1$ to $L_2$ by repeating it with time reversed, so that the negative crossing just made
is cancelled by the positive crossing of the inverse homotopy.
The resulting tangle $L_3$ is same as $L_1$.
Repeat such a cancelling pair of homotopies along each of the segments $t_2\x I,\dots,t_{n-1}\x I$, thus obtaining
tangles $L_4,\dots,L_{2n-1}$ so that $L_5,L_7,\dots,L_{2n-1}$ are the same as $L_1$ and $L_3$.
Finally perform the original crossing of $h_t$ at $q$ to get the tangle $L_{2n}=h_{s+\eps}$.
It remains to observe that for each $i=0,\dots,n$, the combination of the positive crossing along $t_i\x I$ 
(from $L_{2i}$ to $L_{2i+1}$) and the negative crossing along $t_{i+1}\x I$ (from $L_{2i+1}$ to $L_{2i+2}$), 
if done simultaneously, are paired up by a Whitney disk meeting the image of $N$ transversally in one interior point.
This Whitney disk arises as above from a $C_2$-move between $L_{2i}$ and $L_{2i+2}$.
\end{proof}

\begin{corollary} \label{whitney1'}
Two tangles are $\frac12$-quasi-isotopic if and only if they are self $C_2$-equivalent.
\end{corollary}

It is known that two links are $0$-quasi-isotopic, i.e.\ link homotopic, if and only if
they are equivalent up to $C_2$-moves that involve strands from at most two components \cite{NS}.

\begin{remark} \label{unknotting}
Since there are no nontrivial type $1$ invariant of knots, every knot $S^1\emb S^3$ is $C_2$-equivalent to the unknot 
by the simplest case of the Goussarov--Habiro theorem \cite{Gu}, \cite{Ha}.
It also follows from this that every string knot $I\emb I\x\R^2$ is $C_2$-equivalent to the trivial string 
knot $\Xi$.
Another proof of these facts can be obtained from Lemma \ref{whitney1}.
Indeed, every generic null-homotopy of a knot can be made into a $(-\frac12)$-quasi-isotopy by introducing 
simultaneously with every crossing change a first Reidemeister move of the opposite sign.
\end{remark}

\subsection{Classification: string links}

We will use the notation $\eps_z$, $l_z$, $\Delta(h_+)$ from the definition of 
$\Sigma\:LH_{\Lambda,\Lambda'}\to\Z[t^{\pm1}]\oplus\Z[t^{\pm1}]$.

\begin{lemma} \label{delayed crossing}
Let $h$ be a generic link homotopy between two-component string links $\Lambda$ and $\Lambda'$.
Suppose that $h$ has precisely two double points $z_1$, $z_2$, where $z_1$ occurs before $z_2$.
Then $\Lambda$ and $\Lambda'$ are joined by a generic link homotopy $h'$ with precisely two double points
$w_1$, $w_2$ such that $w_i$ occurs on the same component as $z_i$, $\eps_{w_i}=\eps_{z_i}$ and
$l_{w_i}=l_{z_i}$ for each $i=1,2$, but $w_1$ occurs after $w_2$.
\end{lemma}

The idea of proof is to ``delay'' one intersection until another one occurs. 
To this end, we grow a ``tendril'', which is eliminated after the delayed intersection is finally performed.
The author learned this technique from P. Akhmetiev; the word ``tendril'' is also his (see \cite{AR}).

\begin{proof} Let us write $h$ as $h_t\:I\sqcup I\to I\x\R^2$.
We have $z_1=h_{s_1}(x_1)=h_{s_1}(y_1)$ and $z_2=h_{s_2}(x_2)=h_{s_2}(y_2)$, where $s_1<s_2$.
Then $h_t$ for $t\in [s_1+\eps,\,s_2-\eps]$ is a smooth isotopy, so it is covered by an ambient isotopy
$H_t\:I\x\R^2\to I\x\R^2$, so that, in particular, $H_0=\id$ and $H_1 h_{s_1+\eps}=h_{s_2-\eps}$.
Let $J_1$ be a short arc in $I\x\R^2$ connecting $h_{s_1+\eps}(x_1)$ and $h_{s_1+\eps}(y_1)$ and otherwise
disjoint from $h_{s_1+\eps}(I\sqcup I)$.
Then $H_1(J_1)$ is an arc meeting $h_{s_2-\eps}$ only in its endpoints.
Let $J_2$ be a short arc in $I\x\R^2$ connecting $h_{s_2-\eps}(x_2)$ and $h_{s_2-\eps}(y_2)$ and otherwise
disjoint from $h_{s_2-\eps}(I\sqcup I)$ and from $H_1(J_1)$.
Since $H_1$ is uniformly continuous, some neighborhood $N_2$ of $J_2$ is disjoint from $H_1(N_1)$ for some 
neighborhood $N_1$ of $J_1$.

Then $h_{s_1-\eps}$ is isotopic by a smooth isotopy $g^1_t$ to a string link $\Lambda_1$ that is link homotopic 
to $h_{s_1+\eps}$ by a generic homotopy $f^1_t$ with support in $N_1$.
Moreover, we may assume that $f^1_t$ has precisely one double point $v_1$ that occurs in the same component
as $z_1$ and satisfies $\eps_{v_1}=\eps_{z_1}$ and $l_{v_1}=l_{z_1}$.

Similarly, $h_{s_2-\eps}$ is link homotopic by a generic homotopy $f^2_t$ with support in $N_2$ to a string link 
$\Lambda_2$ that is isotopic to $h_{s_2+\eps}$ by a smooth isotopy $g^2_t$.
Moreover, we may assume that $f^2_t$ has precisely one double point $v_2$ that occurs in the same component
as $z_2$ and satisfies $\eps_{v_2}=\eps_{z_2}$ and $l_{v_2}=l_{z_2}$.

It remains to replace the time interval $[s_1-\eps,\,s_2+\eps]$ of the link homotopy $h_t$ by the following
sequence of link homotopies: (i) $g^1_t$; (ii) $H_t\Lambda_1$; (iii) $f^2_t$ in $N_2$ and identity elsewhere; 
(iv) $H_1f^1_t$ in $H_1(N_1)$ and identity elsewhere; (v) $g^2_t$.
\end{proof}

\begin{lemma} \label{whitney permutation}
Let $h=h_+\sqcup h_-$ be a generic link homotopy between two-component string links $\Lambda$ and $\Lambda'$.
Suppose that $h$ has precisely two double points $z$ and $w$, both in $\Delta(h_+)$, such that $l_z=l_w$
and $\eps_z=-\eps_w$.
Then $\Lambda$ and $\Lambda'$ are $\frac12$-quasi-isotopic.
\end{lemma}

\begin{proof} Let us first consider the case where some (or equivalently each) time instant of $h_-$ is
unknotted.
We may assume by symmetry that $z$ occurs before $w$.
Arguing as in Lemma \ref{delayed crossing}, we may delay $z$ until $w$ so as to make the two crossings
simultaneously.
Let $f_t=f_t^+\sqcup f_t^-$ be the resulting link homotopy and let $f_s\:I\sqcup I\to I\x\R^2$ be its singular instant; 
its two double points, still denoted $z$ and $w$, satisfy $l_z=l_w$ and $\eps_z=-\eps_w$.
Here $l_z$ and $l_w$ are defined using the arcs $J_z,J_w\subset I$ such that $f_s^+(\partial J_z)=z$ and 
$f_s^+(\partial J_w)=w$.
Namely, $l_z$ and $l_w$ are the classes of $h_s^+(J_z)$ and $h_s^+(J_w)$ in $H_1\big(I\x\R^2\but h_s^-(I)\big)$.

If $J_w\subset J_z$, then $[h_s^+(\Cl{J_z\but J_w})]=0\in H_1\big(I\x\R^2\but h_s^-(I)\big)$.
Since $h_s^-$ is unknotted by our assumption, $\pi_1\big(I\x\R^2\but h_s^-(I)\big)\simeq\Z$, and so 
$h_s^+(\Cl{J_z\but J_w})$ is null-homotopic in $I\x\R^2\but h_s^-(I)$.
Hence $h_s^+$ is a $\frac12$-quasi-tangle.

The case $J_z\subset J_w$ is similar, but if neither of the two inclusions hold, we need to go back in time 
and redefine $f_t$ so as to achieve one of these inclusions.
It suffices to be able to permute a specified endpoint $x$ of $J_z$ with a specified endpoint $y$ of $J_w$,
provided that the other endpoint $x'$ of $J_z$ and the other endpoint $y'$ of $J_w$ do not belong to 
the arc $[x,y]\subset I$.
To this end let us return to the instant $h_{s-\eps}$ for a sufficiently small $\eps>0$, and precede the two
crossings by pulling $h_{s-\eps}^+(N_{x'})$ past $h_{s-\eps}^+(N_{y'})$ along $h^+_{s-\eps}([x,y])$, where 
$N_{x'}$ and $N_{y'}$ are sufficiently small neighborhods of $x'$ any $y'$ in $I$.
Now perform the delayed crossings, simultaneously, one between neighborhoods of $y$ and $y'$ and another 
between neighborhoods of $x'$ and of a point $x''$ that is near $y$ but is not in $[x,y]$.
Luckily, it is easy to see that the resulting link is isotopic to $h_{s+\eps}$.
(See Figure \ref{reorder}.)

\begin{figure}[h]
\includegraphics[width=8cm]{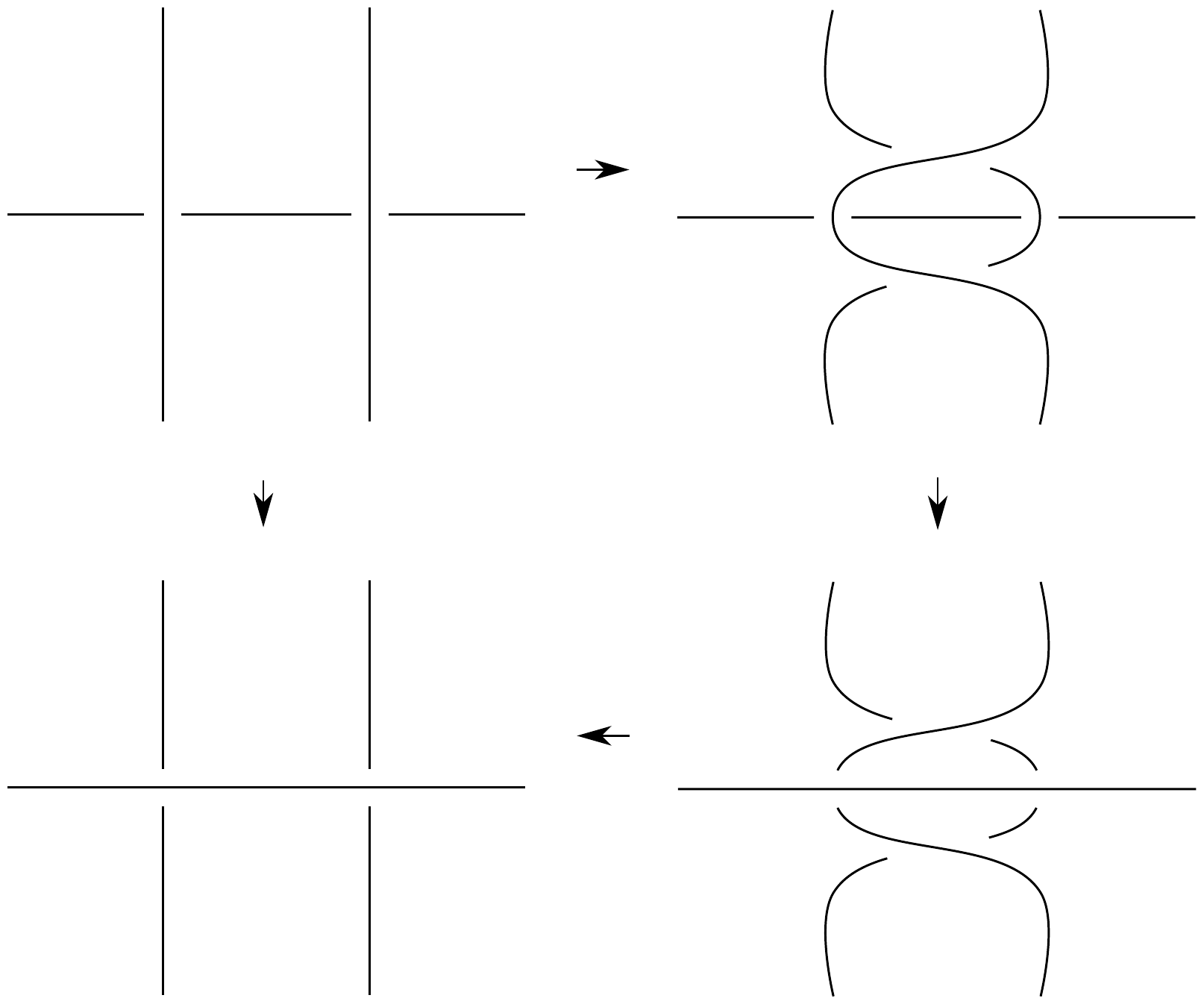}
\caption{Rearranging a pair of simultaneous $C_1$-moves}
\label{reorder}
\end{figure}

Now let us prove the lemma in the general case.
Every knot is $C_2$-equivalent to the unknot (see Remark \ref{unknotting}).
Hence $\Lambda'$ is self $C_2$-equivalent to a string link $\tilde\Lambda'=\tilde\Lambda'_+\sqcup\tilde\Lambda'_-$ 
such that $\tilde\Lambda'_-$ is unknotted. 
Let $g$ be a $\frac12$-quasi-isotopy between $\Lambda'$ and $\tilde\Lambda'$, and let $1_{\tilde\Lambda'}$ be 
the identical self-homotopy of $\tilde\Lambda'$.
Then $h':=h*g*1_{\tilde\Lambda'}$ is a link homotopy between $\Lambda$ and $\tilde\Lambda'$.

By Lemma \ref{delayed crossing} we may delay the two crossings of $h'$ until $1_{\tilde\Lambda'}$; 
thus there exist a $\frac12$-quasi-isotopy $g'$ from $\Lambda$ to some string link $\tilde\Lambda$
and a generic link homotopy $f=f_+\sqcup f_-$ from $\tilde\Lambda$ to $\tilde\Lambda'$ with precisely 
two double points $z'$ and $w'$, both in $\Delta(f_+)$, such that $l_{z'}=l_{w'}$ and $\eps_{z'}=-\eps_{w'}$.
By the above, $\tilde\Lambda$ and $\tilde\Lambda'$ are joined by a $\frac12$-quasi-isotopy $f'$.
Then $g'*f'*\bar g$ is $\frac12$-quasi-isotopy from $\Lambda$ to $\Lambda'$.
\end{proof}

\begin{theorem} \label{NOstring}
Two-component string links $\Lambda$ and $\Lambda'$ are self $C_2$-equivalent if and only if 
$\lk(\Lambda)=\lk(\Lambda')$ and $\beta_i(\Lambda)=\beta_i(\Lambda')$ for $i=0,1$.
\end{theorem}

A slightly different classification of two-component string links up to self $C_2$-equivalence was obtained by 
Fleming and Yasuhara \cite{FY07}.

\begin{proof} The ``only if'' assertion is clear.
Conversely, since $\lk(\Lambda)=\lk(\Lambda')$, $\Lambda$ and $\Lambda'$ are link homotopic.
Then by Corollary \ref{realization1}, they are joined by a link homotopy $h=h_+\sqcup h_-$ such that $\Sigma(h)=0$.
Thus $\Delta(h_+)$ is partitioned into pairs $(z,z')$ such that $l_z=l_{z'}$ and $\eps_z=-\eps_{z'}$; and similarly
for $\Delta(h_-)$. 
By Lemma \ref{delayed crossing} we may assume that for every such pair $(z,z')$, $z'$ occurs immediately after $z$
in the homotopy, with no other double points in between.
Then by Lemma \ref{whitney permutation} $\Lambda$ is $\frac12$-quasi-isotopic to $\Lambda'$.
Hence by Corollary \ref{whitney1'} $\Lambda$ is self $C_2$-equivalent to $\Lambda'$.
\end{proof}

\subsection{Classification: links}

The following lemma follows by the proof of Lemma \ref{delayed crossing}.

\begin{lemma} \label{delayed crossing2}
Let $H$ be a generic link homotopy between two-component links $L$ and $L'$.
Suppose that $H$ has precisely two double points $z_1$, $z_2$, where $z_1$ occurs before $z_2$.
Then $L$ and $L'$ are joined by a generic link homotopy $H'$ with precisely two double points
$w_1$, $w_2$ such that $w_i$ occurs on the same component as $z_i$, $\eps_{w_i}=\eps_{z_i}$ and
$|l_{z_i}-\frac\lambda2|=|l_{w_i}-\frac\lambda2|$ for each $i=1,2$, but $w_1$ occurs after $w_2$.
\end{lemma}

\begin{lemma} \label{whitney permutation2}
Let $H=H_+\sqcup H_-$ be a generic link homotopy between two-component links $L$ and $L'$, $\lk(L)=\lambda$.
Suppose that $H$ has precisely two double points $z$ and $w$, both in $\Delta(H_+)$, such that 
$|l_z-\frac\lambda2|=|l_w-\frac\lambda2|$ and $\eps_z=-\eps_w$.
Then $L$ and $L'$ are $\frac12$-quasi-isotopic.
\end{lemma}

\begin{proof} This is similar to the proof of Lemma \ref{whitney permutation}, except that we now need to
improve the hypothesis $|l_z-\frac\lambda2|=|l_w-\frac\lambda2|$ into $l_z-\frac\lambda2=l_w-\frac\lambda2$.
This can be done as long as we are free to choose between $l_z$ and $l'_z$.
But such a choice is indeed possible, since we can freely slide endpoints of $J_z$ past endpoints of $J_w$,
by the proof of Lemma \ref{whitney permutation}.
\end{proof}

\begin{theorem} {\rm (Nakanishi--Ohyama \cite{NO})} \label{NO}
Two-component links $L$ and $L'$ are self $C_2$-equivalent if and only if 
$\lk(L)=\lk(L')$ and $\beta(L)=\beta(L')$.
\end{theorem}

\begin{proof} This follows from Corollaries \ref{realization2} and \ref{whitney1'} and 
Lemmas \ref{delayed crossing2} and \ref{whitney permutation2} similarly to the proof of Theorem \ref{NOstring}.
\end{proof}

\subsection*{Acknowledgements}

The author is grateful to P.~ Akhmetiev, A.~ Gaifullin, Y.~ Kotorii, A.~ Lightfoot, A.~ Yasuhara and the referee 
for useful comments.


\begin{thebibliography}{00}

\bib{AR}{article}{
   author={Akhmetiev, P. M.},
   author={Repov\v s, D.},
   title={A generalization of the Sato--Levine invariant},
   journal={Tr. Mat. Inst. Steklova},
   volume={221},
   date={1998},
   pages={69--80},
   translation={
      journal={Proc. Steklov Inst. Math.},
      date={1998},
      number={221},
      pages={60--70},
   },
}

\bib{BT}{article}{
   author={Bartels, A.},
   author={Teichner, P.},
   title={All two-dimensional links are null homotopic},
   journal={Geom. Topol.},
   volume={3},
   date={1999},
   pages={235--252},
   note={\arxiv[GT]{0004021}},
}

\bib{FR}{article}{
   author={Fenn, R.},
   author={Rolfsen, D.},
   title={Spheres may link homotopically in $4$-space},
   journal={J. London Math. Soc. (2)},
   volume={34},
   date={1986},
   number={1},
   pages={177--184},
}

\bib{FY07}{article}{
   author={Fleming, T.},
   author={Yasuhara, A.},
   title={Milnor numbers and the self delta classification of 2-string links},
   conference={
      title={Intelligence of Low Dimensional Topology 2006},
   },
   book={
      series={Ser. Knots Everything},
      volume={40},
      publisher={World Sci. Publ., Hackensack, NJ},
   },
   date={2007},
   pages={27--34},
}

\bib{FY}{article}{
   author={Fleming, T.},
   author={Yasuhara, A.},
   title={Milnor's invariants and self $C_k$-equivalence},
   journal={Proc. Amer. Math. Soc.},
   volume={137},
   date={2009},
   pages={761--770},
}

\bib{Gu}{article}{
   author={Gusarov, M. N.},
   title={Variations of knotted graphs. The geometric technique of $n$-equivalence},
   journal={Algebra i Analiz},
   volume={12},
   date={2000},
   number={4},
   pages={79--125},
   translation={
      journal={St. Petersburg Math. J.},
      volume={12},
      date={2000},
      pages={569--604},
   },
}

\bib{Hab}{article}{
   author={Habegger, N.},
   title={Applications of Morse theory to link theory},
   conference={
      title={Knots 90},
      address={Osaka},
      date={1990},
   },
   book={
      publisher={de Gruyter, Berlin},
   },
   date={1992},
   pages={389--394},
}

\bib{Ha}{article}{
   author={Habiro, K.},
   title={Claspers and finite type invariants of links},
   journal={Geom. Topol.},
   volume={4},
   date={2000},
   pages={1--83},
}

\bib{HL}{article}{
   author={Habegger, N.},
   author={Lin, Xiao-Song},
   title={The classification of links up to link-homotopy},
   journal={J. Amer. Math. Soc.},
   volume={3},
   date={1990},
   number={2},
   pages={389--419},
}

\bib{KMT}{article}{
   author={Kanenobu, M.},
   author={Miyazawa, Y.},
   author={Tani, A.},
   title={Vassiliev link invariants of order three},
   journal={J. Knot Theory Ramifications},
   volume={7},
   date={1998},
   number={4},
   pages={433--462},
}

\bib{Ki}{article}{
   author={Kirk, P. A.},
   title={Link maps in the four sphere},
   conference={
      title={Differential topology},
      address={Siegen},
      date={1987},
   },
   book={
      series={Lecture Notes in Math.},
      volume={1350},
      publisher={Springer, Berlin},
   },
   date={1988},
   pages={31--43},
}

\bib{Ki2}{article}{
   author={Kirk, P. A.},
   title={Link homotopy with one codimension two component},
   journal={Trans. Amer. Math. Soc.},
   volume={319},
   date={1990},
   number={2},
   pages={663--688},
}

\bib{KL}{article}{
   author={Kirk, P.},
   author={Livingston, C.},
   title={Vassiliev invariants of two component links and the Casson--Walker invariant},
   journal={Topology},
   volume={36},
   date={1997},
   pages={1333--1353},
}

\bib{Ko89}{article}{
   author={Koschorke, U.},
   title={Semicontractible link maps and their suspensions},
   conference={
      title={Algebraic topology Pozna\'n 1989},
   },
   book={
      series={Lecture Notes in Math.},
      volume={1474},
      publisher={Springer, Berlin},
   },
   date={1991},
   pages={160--169},
}

\bib{Ko90}{article}{
   author={Koschorke, U.},
   title={On link maps and their homotopy classification},
   journal={Math. Ann.},
   volume={286},
   date={1990},
   number={4},
   pages={753--782},
}

\bib{Le}{article}{
   author={Levine, J. P.},
   title={An approach to homotopy classification of links},
   journal={Trans. Amer. Math. Soc.},
   volume={306},
   date={1988},
   pages={361--387},
}

\bib{Liv}{article}{
   author={Livingston, C.},
   title={Enhanced linking numbers},
   journal={Amer. Math. Monthly},
   volume={110},
   date={2003},
   number={5},
   pages={361--385},
}

\bib{Matv}{article}{
   author={Matveev, S. V.},
   title={Generalized surgeries of three-dimensional manifolds and
   representations of homology spheres},
   language={Russian},
   journal={Mat. Zametki},
   volume={42},
   date={1987},
   number={2},
   pages={268--278, 345},
   translation={
      journal={Math. Notes},
      volume={42},
      date={1987},
      pages={651--656},
   }
}

\bib{MY}{article}{
   author={Meilhan, J.-B.},
   author={Yasuhara, A.},
   title={Characterization of finite type string link invariants of degree $<5$},
   journal={Math. Proc. Cambridge Philos. Soc.},
   volume={148},
   date={2010},
   pages={439--472},
   note={\arxiv{0904.1527}}
}

\bib{MR1}{article}{
   author={Melikhov, S. A.},
   author={Repov\v s, D.},
   title={$n$-Quasi-isotopy. I. Questions of nilpotence},
   journal={J. Knot Theory Ramifications},
   volume={14},
   date={2005},
   number={5},
   pages={571--602},
   note={\arxiv[GT]{0103114}}
}

\bib{M0}{article}{
   author={Melikhov, S.},
   title={Pseudo-homotopy implies homotopy for singular links of codimension $\ge3$},
   journal={Uspekhi Mat. Nauk},
   volume={55},
   date={2000},
   number={3},
   pages={183--184},
   translation={
      journal={Russ. Math. Surv.},
      volume={55},
      date={2000},
      pages={589--590},
   },
}

\bib{M2}{article}{
   author={Melikhov, S. A.},
   title={Singular link concordance implies link homotopy in codimension $\ge 3$},
   note={(Preprint, 1998); \arxiv{1810.08299}},
}

\bib{M}{article}{
   author={Melikhov, S. A.},
   title={Colored finite type invariants and a multi-variable analogue of the Conway polynomial},
   note={\arxiv[GT]{0312007}},
}

\bib{Mi1}{article}{
   author={Milnor, J.},
   title={Link groups},
   journal={Ann. of Math. (2)},
   volume={59},
   date={1954},
   pages={177--195},
}

\bib{MN}{article}{
   author={Murakami, H.},
   author={Nakanishi, Y.},
   title={On a certain move generating link-homology},
   journal={Math. Ann.},
   volume={284},
   date={1989},
   number={1},
   pages={75--89},
}

\bib{NO}{article}{
   author={Nakanishi, Y.},
   author={Ohyama, Y.},
   title={Delta link homotopy for two component links. III},
   journal={J. Math. Soc. Japan},
   volume={55},
   date={2003},
   number={3},
   pages={641--654},
}

\bib{NS}{article}{
   author={Nakanishi, Y.},
   author={Shibuya, T.},
   title={Link homotopy and quasi self delta-equivalence for links},
   journal={J. Knot Theory Ramifications},
   volume={9},
   date={2000},
   pages={683--691},
}

\bib{Ni}{article}{
   author={Nikkuni, R.},
   title={Homotopy on spatial graphs and generalized {S}ato-{L}evine invariants},
   journal={Rev. Mat. Complut.},
   volume={23},
   date={2010},
   pages={1--17},
}

\bib{Sa}{article}{
   author={Sayakhova, R. F.},
   title={Homotopy classification of singular links of type $(1,1,m;3)$ with $m>1$},
   journal={Zap. Nauchn. Sem. POMI}
   volume={261},
   date={1999},
   number={4},
   pages={229--239, 271},
   translation={
      journal={J. Math. Sci. (New York)},
      volume={110},
      date={2002},
      number={4},
      pages={2886--2891},
   },
}

\bib{ST}{article}{
author={R. Schneiderman},
author={P. Teichner},
title={The group of disjoint 2-spheres in 4-space},
note={\arxiv{1708.00358}},
}

\bib{Ya}{article}{
   author={Yasuhara, A.},
   title={Self delta-equivalence for links whose Milnor's isotopy invariants vanish},
   journal={Trans. Amer. Math. Soc.},
   volume={361},
   date={2009},
   number={9},
   pages={4721--4749},
}

\bib{Ya2}{article}{
   author={Yasuhara, A.},
   title={Classification of string links up to self delta-moves and concordance},
   journal={Algebr. Geom. Topol.},
   volume={9},
   date={2009},
   pages={265--275},
}


\end{thebibliography}
\end{document}